\documentclass[a4paper]{amsart}
\usepackage{fullpage}
\usepackage[english]{babel}
\usepackage[utf8x]{inputenc}
\usepackage{amsmath,amssymb,amsthm}
\usepackage[table]{xcolor}
\usepackage{graphicx,enumerate,url}

\usepackage{tikz,subcaption}
\usetikzlibrary{math,calc,through,backgrounds,shapes,matrix,decorations.markings,decorations.pathmorphing,arrows.meta}
\usetikzlibrary{matrix}
\usetikzlibrary{shapes,decorations.pathmorphing, decorations.markings, calc}
\usetikzlibrary{arrows.meta}
\makeatletter
\tikzset{
  a/.style={
    double=yellow!10,
    draw=teal!50,
  },
  foo /.tip={Stealth[inset=0pt, length=3pt, width=4pt, fill=yellow!10]},
  pics/arc arrow/.style args={#1:#2:#3}{
    code={
      \draw[a, -foo] (0, 0) arc (#1:#2:#3) coordinate (arc@temp);
      \path (arc@temp) ++(270+#2:4pt) ++(#2:3pt) coordinate (-a);
    }
  },
  pics/straight arrow/.style={
    code={
      \begin{scope}[rotate=#1]
        \draw[a, foo-foo] (0, -0.5) -- (0, 0.5);
        \fill[yellow!10] (-1pt, 0.5cm-4.1pt) rectangle (1pt, 0.5cm-3.9pt+\pgflinewidth);
        \fill[yellow!10] (1pt, -0.5cm+4.1pt) rectangle (-1pt, -0.5cm+3.9pt-\pgflinewidth);
        \coordinate (-a) at (3pt, -0.5cm + 4pt);
      \end{scope}
    }
  },
  mydash/.style={dash pattern=on 2mm off 0.5mm}
}
\makeatother

\usepackage{array}
\newcolumntype{C}[1]{>{\centering\arraybackslash$}p{#1}<{$}}
\usepackage{multirow}
\usepackage{multicol}
\usepackage{hhline}
\newcommand{\one}{{e}}
\newcommand{\X}{{\mathcal X}}

\usepackage{hyperref}
\usepackage{enumitem}
\setenumerate{label=\normalfont{(\alph*)}, itemsep=5pt, topsep=5pt}

\definecolor{darkblue}{rgb}{0,0,0.7}

\definecolor{lightblue}{rgb}{0.68,0.85,1}

\definecolor{lightgrey}{rgb}{0.9,0.9,0.9}

\definecolor{grey}{rgb}{0.5,0.5,0.5}

\usepackage{url, hypcap}
\hypersetup{colorlinks=true, citecolor=darkblue, linkcolor=darkblue}

\numberwithin{equation}{section}

\usepackage{cleveref}

\numberwithin{equation}{section}
\numberwithin{figure}{section}

\newtheorem{theorem}{Theorem}[section]
\newtheorem{proposition}[theorem]{Proposition}
\newtheorem{lemma}[theorem]{Lemma}
\newtheorem{corollary}[theorem]{Corollary}
\newtheorem{definition}[theorem]{Definition}

\theoremstyle{definition}
\newtheorem{example}[theorem]{Example}
\newtheorem{remark}[theorem]{Remark}

\usepackage{cite}
\usepackage{color}

\newcommand{\Cov}{\operatorname{Cov}}

\def\rk{\operatorname{rk}}

\def\NN{\mathbb{N}}

\def\RR{\mathbb{R}}
\def\RRd1{\mathbb{R}^{d+1}}

\newcommand{\lspan}{\operatorname{span}}

\usepackage{verbatim}

\def\cN{\mathcal{N}}

\usepackage{wrapfig}

\newcommand{\Dfn}[1]{\emph{\bfseries #1}}
\newcommand{\Var}[1]{{\mathbb{V}[#1]}}

\newcommand{\Sym}{\mathcal{S}}
\newcommand{\inv}{\operatorname{inv}}
\newcommand{\Inv}{\operatorname{Inv}}

\newcommand{\des}{\operatorname{des}}
\newcommand{\Des}{\operatorname{Des}}
\newcommand{\ord}{\operatorname{ord}}
\newcommand{\Exp}[1]{{\mathbb{E}[#1]}}
\newcommand{\Pro}[1]{{\mathbb{P}[#1]}}

\newcommand{\height}{\operatorname{ht}}
\newcommand{\Nroot}[1]{{[#1]}}
\newcommand{\Oroot}[1]{{[#1]}}
\newcommand{\Proot}[1]{{[\widetilde{#1}]}}

\newcommand{\Pinv}{\X_{P_{\leq d}}}
\newcommand{\Ninv}{\X_{N_{\leq d}}}
\newcommand{\Oinv}{\X_{O_{\leq d}}}

\newcommand{\W}[4]{W^{(#1#2)}(#3,#4)}

\newcommand{\dinv}{\operatorname{inv}_d}
\newcommand{\ddes}{\operatorname{des}_d}

\newcommand{\Phiplus}{\Phi^+}
\newcommand{\Phiddes}[1]{\Phi_{\operatorname{des}}^{(#1)}}
\newcommand{\Phidinv}[1]{\Phi_{\operatorname{inv}}^{(#1)}}

\newcommand{\sref}{\mathcal{S}}
\newcommand{\rref}{\mathcal{R}}

\newcommand{\ie}{\textit{i.e.}}

\title[CLTs for generalized descents and inversions]{Central limit theorems for generalized descents and generalized inversions in finite root systems}

\author[K.~Meier]{Kathrin Meier}
\address[K.~Meier \& C.~Stump]{Fakultät für Mathematik, Ruhr-Universit\"at Bochum, Germany}
\email{\{kathrin.meier,christian.stump\}@rub.de}

\author[C.~Stump]{Christian Stump}

\begin{document}

\begin{abstract}
  We consider generalized inversions and descents in finite Weyl groups.
  We establish Coxeter-theoretic properties of indicator random variables of positive roots such as the covariance of two such indicator random variables.
  We then compute the variances of generalized inversions and descents in classical types.
  We finally use the dependency graph method to prove central limit theorems for general antichains in root posets and in particular for generalized descents, and then for generalized inversions.
\end{abstract}

\maketitle

\section{Introduction}
\label{sec:intro}

Permutation statistics have been studied in great detail.
Two of the most prominent examples are the following:
An \Dfn{inversion} of a permutation $\sigma \in \mathcal S_n$ is a pair $(ij)$ with $1 \leq i < j \leq n$ and $\sigma(i) > \sigma(j)$, and a \Dfn{descent} is an inversion of distance~$1$, this is, an inversion $(ij)$ with $j-i = 1$.
The \Dfn{inversion statistic} and the \Dfn{descent statistic} on permutations then count inversions and descents, respectively,
\[
\inv(\sigma) = \big| \{ 1 \leq i < j \leq n \mid \sigma(i) > \sigma(j)\} \big|, \quad
\des(\sigma) = \big| \{ 1 \leq i \leq n-1 \mid \sigma(i) > \sigma(i+1)\} \big|.
\]
The associated random variables $\X_{\inv}$ and $\X_{\des}$ are obtained by asking how many inversions or, respectively, descents does a permutation drawn uniformly at random from~$\mathcal S_n$ have.
Both of these random variables are well-studied, their mean values and variances are known and both satisfy central limit theorems, see for example~\cite[Examples~5.3 \& 5.5]{bender1973central}.

\medskip

Generalizations of these classical results towards various directions have been considered.
The aim of this paper is to combine the follwoing two considerations.

\medskip

One may define a $d$\Dfn{-inversion} as an inversion of distance at most~$d$, and a $d$\Dfn{-descent} to be an inversion of distance exactly~$d$.
The $d$\Dfn{-inversion statistic} and the $d$\Dfn{-descent statistic} then count $d$-inversions and $d$-descents, respectively,
\begin{align*}
  \dinv(\sigma) &= \big| \{ 1 \leq i < j \leq \min\{i+d,n\} \mid \sigma(i) > \sigma(j)\} \big|, \\
  \ddes(\sigma) &= \big| \{ 1 \leq i \leq n-d \mid \sigma(i) > \sigma(i+d)\} \big|.
\end{align*}
The variance for the corresponding random variable $\X_{\dinv}$ have been computed in~\cite[Lemma~1]{bona_2008} for $n \geq 2d$ and in~\cite[Theorem~1]{Pike_2011} for all parameters $n,d$.
Both papers then deduce central limit theorems for $d$-inversions\footnote{In both references, this statistic was called $d$-descents. We decided to use the name $d$-decents for the other variant of generalization which appears to not have been previously considered in the context of their asymptotic behaviour.} for fixed~$d$ in~\cite[Theorem~2]{bona_2008} and for general~$d = d_n$ as a function of~$n$ in~\cite[Theorem~4]{Pike_2011}.
The latter also obtains rates of convergence in the cases that $d_n$ grows faster than~$n^{2/3}$ or slower than~$n^{1/3}$.

\medskip

The second generalization is towards finite Coxeter groups, we refer to Section~\ref{sec:rootsystems} for definitions.
In~\cite[Theorems~3.1 \&~4.1]{Kahle_2019}, a random variable~$\X_\beta$ for any root~$\beta$ in the root system of a finite Coxeter group is defined.
These variables are then used to uniformly express the mean values and variances of the random variables for the descent and inversion statistics in terms of root system data.
In addition,~\cite[Theorems~ 6.1 \& 6.2]{Kahle_2019} give necessary and sufficient conditions on sequences of finite Coxeter groups for which these statistics satisfy central limit theorems.

\medskip

We study in this paper a natural way to jointly generalize these two contexts.
The defined notions of $d$-inversions and $d$-descents for permutations have natural analogues for finite Weyl groups, or rather for finite \emph{crystallographic} root systems using the combinatorial notion of \emph{root posets}.
We generalize the above results by exhibiting structural interactions between finite Coxeter groups and probabilistic notions.
Most notably, we show that
\begin{itemize}
  \item the two random variables~$\X_\beta$ and~$\X_\gamma$ for two roots $\beta,\gamma$ are independent if and only if~$\beta$ and~$\gamma$ are orthogonal, and that
  \item for any set~$\Psi$ of positive roots, the graph on~$\Psi$ with edges between non-orthogonal roots is a dependency graphs of the random variable~$\sum_{\beta \in \Psi}\X_\beta$.
\end{itemize}
Using these interactions, we obtain central limit theorems in this generalized context.
Moreover, we obtain rates of convergence for certain growth behavior of the parameter~$d = d_n$ which generalize the above cases for the symmetric group.

\medskip

The paper is organized as follows.
We give the necessary background on the combinatorics of finite crystallographic root systems and their Weyl groups in Section~\ref{sec:rootsystems} and then proceed with presenting the main results of this paper in Section~\ref{sec:mainresults}.
Theorem~\ref{thm:antichain} is a central limit theorem for antichains in root posets together with a rate of convergence.
This in particular covers the situation of generalized descents in Corollary~\ref{cor:gendesc}.
Theorem~\ref{thm:geninv} is then the analogous result for generalized inversions and Corollary~\ref{cor:geninv} gives rates of convergence where our method is applicable.
Section~\ref{sec:randominvs} then uses the combinatorics of root systems to uniformly describe the covariance between indicator random variables for positive roots in root systems.
Section~\ref{section:CLT} uses all previously gathered information and also explicit calculations from the appendix to derive central limit theorems where possible using the dependency graph method in the version introduced in \cite{janson_1988}.
Section~\ref{sec:concretevariances} finally uses the combinatorial realizations of root systems in classical types to explicitly compute the variances in all considered situations.

\section*{Acknowledgements}

We would like to thank Valentin F\'eray, Thomas Kahle and Christoph Th\"ale for helpful discussions, and two anonymous referees for valuable comments.
They in particular revealed a gap in our application of the dependency graph method which we fixed in Theorem~\ref{thm:CovSet}.

\section{Main results}
\label{sec:main}

In order to state the main results, we recall necessary notions from finite root systems and define the considered random variables.

\subsection{Finite root systems}
\label{sec:rootsystems}
Most definitions in this subsection can be found in standard references such as~\cite[Part~I]{Bjoerner_2006}\footnote{There is one caveat: the below definition of root poset is also standard in the combinatorics literature on finite crystallographic root systems---but this notion differs from the notion of root poset in~\cite[Section~4.6]{Bjoerner_2006}.}.
These can also be found in~\cite{Kahle_2019}.

\medskip

Let $\Delta \subseteq \Phi^+ \subset \Phi = \Phi^+ \cup \Phi^- \subseteq V$ be a finite crystallographic root system of rank~$\rk(\Phi) = n = |\Delta|$ inside a Euclidean vector space~$V$ of dimension~$n$ with inner product $\langle \cdot,\cdot\rangle$.
The roots in $\Delta$ are called \Dfn{simple}, those in $\Phi^+$ \Dfn{positive} and those in $\Phi^- = -\Phi^+$ \Dfn{negative}.
Every positive root is an integer linear combination of simple roots.
This induces a natural partial order on~$\Phiplus$ given by the cover relations $\beta \prec \gamma$ for $\gamma - \beta \in \Delta$.
This \Dfn{root poset} is graded by $\height(\beta) = \sum_{\alpha \in \Delta}\lambda_\alpha$ for $\beta = \sum_{\alpha \in \Delta}\lambda_\alpha\alpha \in \Phiplus$.
Its rank is one less than the \Dfn{Coxeter number}~$h$ of the root system.
Root posets have no natural analogue beyond crystallographic types, see~\cite{CS2015}.

\medskip

Let $W \leq \operatorname{O}(n)$ be the corresponding Weyl group generated by the set of (orthogonal) \Dfn{simple reflections} $\sref = \{s_\alpha \mid \alpha \in \Delta \}$.
The set of \Dfn{reflections} in~$W$ is then
\[
\rref = \big\{ w s_\alpha w^{-1} \mid w \in W, s_\alpha \in \sref \big\} = \big\{ s_\beta \mid \beta \in \Phi^+\big\}\ .
\]
Denote by $\one \in W$ the identity element, by $w_\circ \in W$ the unique \Dfn{longest element} defined by the property $w_\circ(\Delta) = -\Delta$ or, equivalently, by $w_\circ(\Phi^+) = \Phi^-$.
The product $s_\beta s_\gamma \in W$ of two reflections corresponding to positive roots $\beta,\gamma \in \Phi^+$ is a rotation and we set
\[
\ord(\beta,\gamma) = \ord(s_\beta s_\gamma) = \min\big\{k > 0\mid (s_\beta s_\gamma)^k = \one \big\}.
\]
The condition to be crystallographic implies that this order can only take values in $\{2,3,4,6\}$.
An \Dfn{inversion} of an element $w \in W$ is a positive root that is sent by~$w$ to a negative root,
\[
\Inv(w) = \Phi^+ \cap w^{-1}(\Phi^-) = \big\{ \beta \in \Phi^+ \mid w(\beta) \in \Phi^-\big\}.
\]
Moreover, a \Dfn{descent} of $w$ is a simple inversion,
\[
\Des(w) = \Delta \cap w^{-1}(\Phi^-) = \big\{ \beta \in \Delta \mid w(\beta) \in \Phi^-\big\}.
\]
\begin{example}
  \label{ex:A}
  The irreducible root system of type $A_{n-1}$ may be realized inside the Euclidean space $\RR^n$ with standard basis $\{e_1,\dots,e_n\}$ as $\Delta \subseteq \Phiplus \subseteq \Phi$ given by
  \[
  \{ e_{i+1} - e_i \mid 1 \leq i < n\} \subseteq \{ e_j - e_i \mid 1 \leq i < j \leq n\} \subseteq \{ e_j - e_i \mid 1 \leq i \neq j \leq n\}\,.
  \]
  We have $e_j - e_i = (e_{i+1}-e_i) + \dots  + (e_j - e_{j-1}) \in \Phiplus$ so $\height(e_j-e_i) = j-i$.
  We abbreviately write $[ij]$ for $e_j - e_i$.
  Its Weyl group is the symmetric group $\Sym_n$ acting on $\RR^n$ by permuting the standard basis.
  The simple reflection are the simple transpositions $s_{[i,i+1]} = (i,i+1)$ for $1 \leq i < n$ and the reflections are all transpositions $s_{[ij]} = (i,j)$ for $1 \leq i < j \leq n$.
  Moreover, we have
  \[
  \ord\big(s_{[ij]}s_{[k\ell]}\big) = \begin{cases}
    1 &\text{ if } |\{i,j\} \cap \{k,\ell\} | = 2\,, \\
    2 &\text{ if } |\{i,j\} \cap \{k,\ell\} | = 0\,, \\
    3 &\text{ if } |\{i,j\} \cap \{k,\ell\} | = 1\,.
  \end{cases}
  \]
  The root poset is, for $n=5$, depicted as
  
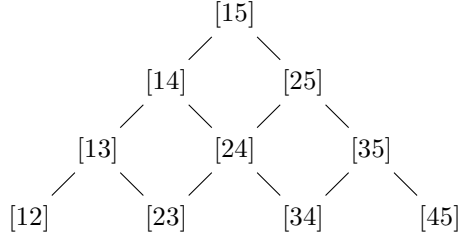
\begin{center}
    \begin{tikzpicture}[scale=0.9]
      \foreach \i in {1,2,3,4} {
        \draw[-] (2*\i+1,1) -- (2+\i,\i);
        \draw[-] (2*\i+1,1) -- (5+\i,5-\i);
      }
      \foreach \i/\j in {1/2,2/3,3/4,4/5,1/3,2/4,3/5,1/4,2/5,1/5}
      \node[circle,fill=white,inner sep=2pt] at (\i+\j,\j-\i) {$\Nroot{\i\j}$};
    \end{tikzpicture}
  \end{center}
\end{example}

\begin{example}
  \label{ex:B}
  The irreducible root system of type $B_n$ may be realized inside the Euclidean space $\RR^n$ with standard basis $\{e_1,\dots,e_n\}$ as $\Delta \subseteq \Phiplus \subseteq \Phi$ given by
  \begin{align*}
    \{e_1\} \, \cup \, \{ e_{i+1} - e_i \mid 1 \leq i < n\} &\subseteq \{\phantom{\pm}e_i \mid 1 \leq i \leq n\} \, \cup \,  \{ e_j \pm e_i \mid 1 \leq i < j \leq n\} \\
    &\subseteq \{\pm e_i \mid 1 \leq i \leq n\} \, \cup \, \{ e_j \pm e_i \mid 1 \leq i \neq j \leq n\}\,.
  \end{align*}
  As in type $A_{n-1}$, we have $e_j - e_i = (e_{i+1}-e_i) + \dots + (e_j - e_{j-1})$.
  Moreover,  $e_i = e_1 + (e_i-e_1)$
  and
  \[
  e_j+e_i = 2 e_1 + \sum\limits_{\ell = 1}^{i-1} (e_{\ell +1} - e_{\ell} ) + \sum\limits_{k = 1}^{j-1} (e_{k +1} - e_{k})\,.
  \]
  This yields
  \[
  \height(e_j-e_i) = j-i, \quad \height(e_i) = i, \quad \height(e_j+e_i) = j+i\,.
  \]
  We abbreviately write $\Nroot{ij}$ for $e_j-e_i$, $\Proot{ij}$ for $e_j+e_i$ and $\Oroot{i}$ for $e_i$.
  Its Weyl group is the group of signed permutations---all bijections from $\{\pm 1 , \dots \pm n\}$ to itself such that $\pi(-i) = -\pi(i)$, again acting on $\RR^n$ by permuting the standard basis.
  The simple reflection are
  \[
  s_\Nroot{i,i+1} = (i,i+1)(-i,-i-1),\quad
  s_\Oroot{1}  = (1,-1)\, ,
  \]
  and all transpositions are
  \[
  s_\Nroot{i,j} = (i,j)(-i,-j), \qquad
  s_\Oroot{i}  = (i,-i) ,\qquad
  s_\Proot{i,j} = (i,-j)(-i,j)
  \]
  Moreover, we have additionally to type A
  \begin{align*}
    \ord\big(s_\Oroot{i}s_\Oroot{j}\big) &= \begin{cases}
      1 &\text{ if } i=j \\
      2 &\text{ if } i \neq j
    \end{cases}
    \quad
    &&\ord\big(s_\Oroot{i}s_\Nroot{k\ell}\big)&& \mkern-22mu = \begin{cases}
      2 &\text{ if } |\{i\} \cap \{k, \ell\} | = 0\\
      4 &\text{ if } |\{i\} \cap \{k, \ell\} | = 1
    \end{cases} \\
    \ord\big(s_\Oroot{i}s_\Proot{k\ell}\big) &= \begin{cases}
      2 &\text{ if } |\{i\} \cap \{k, \ell\} | = 0\\
      4 &\text{ if } |\{i\} \cap \{k, \ell\} | = 1
    \end{cases}
    \quad
    &&\ord\big(s_\Nroot{ij}s_\Proot{k\ell}\big) && \mkern-22mu = \begin{cases}
      2 &\text{ if } |\{i,j\} \cap \{k, \ell\} | \in \{0,2\}\\
      3 &\text{ if } |\{i,j\} \cap \{k, \ell\} | = 1
    \end{cases}\\
    \ord\big(s_\Proot{ij}s_\Proot{k\ell}\big) &= \begin{cases}
      1 &\text{ if } |\{i,j\} \cap \{k, \ell\} | = 2 \\
      2 &\text{ if } |\{i,j\} \cap \{k, \ell\} | = 0 \\
      3 &\text{ if } |\{i,j\} \cap \{k, \ell\} | = 1 
    \end{cases}
  \end{align*}
  
  The root poset is, for $n=4$, depicted as
  
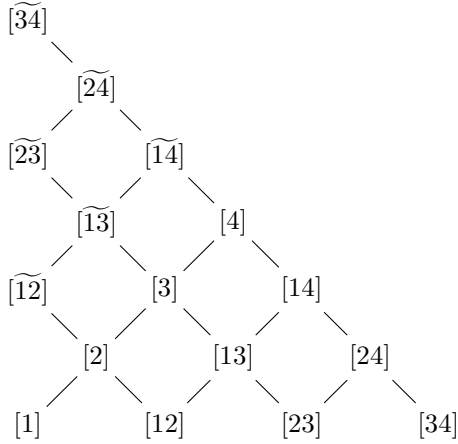
\begin{center}
    \begin{tikzpicture}[scale=0.9]
      \foreach \i in {1,2,3,4}
      \draw[-] (2*\i-1,1) -- (1,2*\i-1) -- (5-\i,3+\i);
      \foreach \i in {2,3}
      \draw[-] (2*\i-1,1) -- (3+\i,5-\i);
      \foreach \i in {1,2,3,4}
      \node[circle,fill=white,inner sep=2pt] at (\i,\i) {$\Oroot{\i}$};
      \foreach \i/\j in {1/2,2/3,3/4,1/3,2/4,1/4}
      \node[circle,fill=white,inner sep=2pt] at (\i+\j,\j-\i) {$\Nroot{\i\j}$};
      \foreach \i/\j in {1/2,2/3,3/4,1/3,2/4,1/4} 
      \node[circle,fill=white,inner sep=2pt] at (\j-\i,\j+\i) {$\Proot{\i\j}$};
    \end{tikzpicture}
  \end{center}
\end{example}

\subsection{Main results}
\label{sec:mainresults}

For $\beta \in \Phiplus$, one considers the Bernoulli random variable on elements chosen uniformly at random in~$W$ indicating whether or not the positive root~$\beta$ is an inversion,
\begin{align}
  \label{def:RV}
  \X_\beta(w) = \begin{cases}
    1 &\text{ if } w(\beta) \in \Phi^-\ , \\
    0 &\text{ if } w(\beta) \in \Phiplus\ .
  \end{cases}
\end{align}
In other words, $\X_\beta$ is the Bernoulli random variable given by
\begin{equation}
  \Pro{\X_\beta = 1} = \frac{\big|\{w \in W \mid w(\beta) \in \Phi^-\}\big|}{|W|}\,, \qquad
  \Pro{\X_\beta = 0} = \frac{\big|\{w \in W \mid w(\beta) \in \Phi^+\}\big|}{|W|}\,.
  \label{eq:Xbetadef}
\end{equation}
Moreover, for a subset $\Psi \subseteq \Phiplus$ we set $\X_\Psi = \sum_{\beta \in \Psi} \X_\beta$.

\medskip

The main results of this paper concern central limits of this random variable for certain natural subsets of positive roots.
Before presenting these, we observe that the cardinality of $\Psi$ must necessarily go to infinity for $\X_\Psi$ to have a central limit.
But the following example shows that this is not the only obstruction.

\begin{example}
  \label{ex:uniform}
  Let $\Phi$ be the root system of type $A_{n-1}$ and consider
  \[
  \Psi^{(n)} = \{e_i - e_1 \mid 1 < i \leq n\} \subseteq \Phiplus
  \]
  of increasing cardinality.
  In this case $\X_{\Psi^{(n)}}(\pi) = |\{i \mid 1 < i \leq n \text{ and } \pi(i) > \pi(1) \}| $ for $\pi \in \Sym_n$.
  This implies that $\X_{\Psi^{(n)}}$ is uniformly distributed for all~$n$ and thus does not have a central limit.
\end{example}

The following definitions are natural generalizations of $d$-inversions and $d$-descents for permutations.

\begin{definition}
  Let $\Phi$ be a finite crystallographic root system with root poset $\Phiplus$ and let $d \in \mathbb{N}$.
  We then set
  \begin{align}
    \Phiddes{d} = \big\{ \beta \in \Phiplus \mid \height(\beta) = d \big\}, \quad \Phidinv{d} = \big\{ \beta \in \Phiplus \mid \height(\beta) \leq d \big\}
    \label{eq:ddesdinv}
  \end{align}
  of positive roots.
  A (generalized) $d$\Dfn{-inversion} of an element $w \in W$ is an inversion of $w$ inside $\Phidinv{d}$ and a (generalized) $d$\Dfn{-descent} is an inversion of $w$ inside $\Phiddes{d}$.
\end{definition}

It can easily be seen, compare Example~\ref{ex:A}, that this notion generalizes the above notions for the symmetric group.
Moreover, the root poset has height $h-1$ and thus $d$-inversions are the same as usual inversions for Weyl groups in the case of $d = h-1$.
Similarly, since the roots having height~$1$ in the root poset are the simple roots~$\Delta$, generalized $d$-descents are the same as usual descents for $d = 1$.
For better readability, we from now on drop the term \emph{generalized}.

\medskip
\begin{example}
  The notion of an inversion of height~$d$ has the following graphical interpretation in all classical types.
  We show this in the following situations in the four classical types.
  In type $A_4$ with $d=3$, we obtain $\Phiddes{3} = \{\Nroot{14}, \Nroot{25}\}$, so we compare entries in the one-line notation of distance~$3$, indicated as
  
  \begin{figure}[h]
    \centering
    \begin{tikzpicture}
      \foreach \i/\j in {1/4, 2/5}
      \draw[->,thick] (\i+0.05,0.35) to [bend left=40] node[above]{\small{$\Nroot{\i\j}$}} (\j-0.05,0.35);
      \foreach \i in {1,2,...,5}
      \node[circle,fill=white,inner sep=2pt] at (\i,0) {${\sigma_{\i}}$};
      \foreach \i in {1.5, ..., 4.5}
      \draw [line width=2pt, opacity=0.1] (\i, 0.5) -- (\i, -0.5);
    \end{tikzpicture}
    \caption*{Example for $A_4$}
  \end{figure}
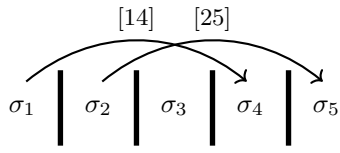
  
  In type $B_4$ with $d=3$, we have $\Phiddes{3} = \{\Proot{12}, \Oroot{3}, \Nroot{14}\}$ and in type $C_4$ with $d=3$, we have $\Phiddes{3} = \{\Oroot{2}, \Proot{13}, \Nroot{14}\}$.
  We thus compare entries of distance~$3$ in the below sense under the additional condition that we only compare $\overline\sigma_i$ with $\sigma_j$ if $i \leq j$.
  This is pictorially indicated as
  \begin{figure}[h]
    \centering
    \begin{tikzpicture}
      \foreach \i/\j in {1/2}
      \draw[->,thick] (0.9*-\i-0.05,0.35) to [bend left=40] node[pos=0.2, above]{\small{$\Proot{\i\j}$}} (0.9*\j-0.05,0.35);
      \foreach \i/\j in {0/3}
      \draw[->,thick] (0.9*\i+0.05,0.35) to [bend left=70] node[above]{\small{$\Oroot{\j}$}} (0.9*\j-0.05,0.35);
      \foreach \i/\j in {1/4}
      \draw[->,thick] (0.9*\i+0.05,0.35) to [bend left=40] node[pos=0.8, above]{\small{$\Nroot{\i\j}$}} (0.9*\j-0.05,0.35);
      \foreach \i in {1,2,...,4}
      \node[circle,fill=white,inner sep=2pt] at (0.9*-\i,0) {$\overline{\sigma_{\i}}$};
      \node[circle,fill=white,inner sep=2pt] at (0,0) {$0$};
      \foreach \i in {1,2,...,4}
      \node[circle,fill=white,inner sep=2pt] at (0.9*\i,0) {${\sigma_{\i}}$};
      \foreach \i in {-3.15, -2.25, ..., 3.6 }
      \draw [line width=2pt, opacity=0.1] (\i, 0.5) -- (\i, -0.5);
    \end{tikzpicture}
    \caption*{Example for $B_4$}
  \end{figure}
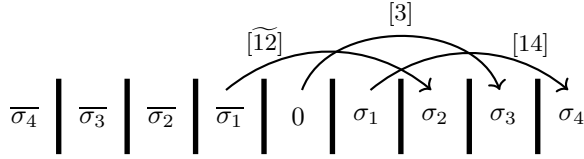
  \begin{figure}[h]
    \centering
    \begin{tikzpicture}
      \foreach \i/\j in {1/2}
      \draw[->,thick] (0.9*-\i-0.05,0.35) to [bend left=70] node[above]{\small{$\Proot{\i3}$}} (0.9*\j-0.05,0.35);
      \foreach \i/\j in {-2/1}
      \draw[->,thick] (0.9*\i+0.05,0.35) to [bend left=40] node[pos=0.2, above]{\small{$\Oroot{2}$}} (0.9*\j-0.05,0.35);
      \foreach \i/\j in {1/4}
      \draw[->,thick] (0.9*\i+0.9*0.05-0.9*1,0.35) to [bend left=40] node[pos=0.8, above]{\small{$\Nroot{\i\j}$}} (0.9*\j-0.9*1.05,0.35);
      \foreach \i in {1,2,...,4}
      \node[circle,fill=white,inner sep=2pt] at (0.9*-\i,0) {$\overline{\sigma_{\i}}$};
      \foreach \i in {1,2,...,4}
      \node[circle,fill=white,inner sep=2pt] at (0.9*\i-0.9*1,0) {${\sigma_{\i}}$};
      \foreach \i in {-3.15, -2.25, ..., 2.7 }
      \draw [line width=2pt, opacity=0.1] (\i, 0.5) -- (\i, -0.5);
    \end{tikzpicture}
    \caption*{Example for $C_4$}
  \end{figure}
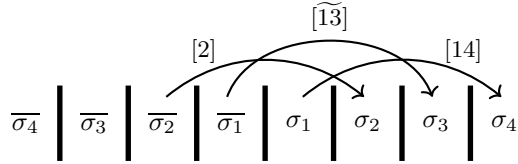
  
  In type $D_4$ with $d=3$, we have $\Phiddes{3} = \{\Proot{14}, \Proot{23}, \Nroot{14}\}$.
  We thus compare entries of distance $3$ in the below sense under the same condition as for types~$B$ and~$C$.
  This is pictorially indicated as
  
  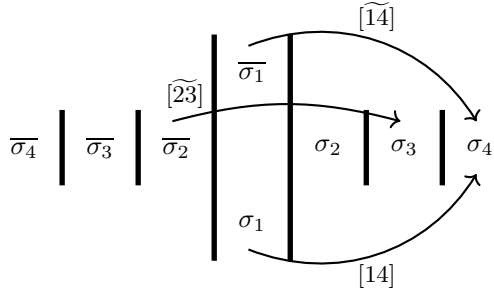
\begin{figure}[h]
    \centering
    \begin{tikzpicture}
      \draw[->,thick] (-1.05,0.35) to [bend left=15] node[pos=0.05, above]{\small{$\Proot{23}$}} (2-0.05,0.35);
      \draw[->,thick] (-0.05,1.35) to [bend left=40] node[above]{\small{$\Proot{14}$}} (3-0.05,0.35);
      \draw[->,thick] (-0.05,-1.35) to [bend left=-40] node[below]{\small{$\Nroot{14}$}} (3-0.05,-0.35);
      \node[circle,fill=white,inner sep=2pt] at (0,1) {$\overline{\sigma_{1}}$};
      \node[circle,fill=white,inner sep=2pt] at (0,-1) {$\sigma_{1}$};
      \foreach \i in {2,3,4}
      \node[circle,fill=white,inner sep=2pt] at (-\i+1,0) {$\overline{\sigma_{\i}}$};
      \foreach \i in {2,3,4}
      \node[circle,fill=white,inner sep=2pt] at (\i-1,0) {${\sigma_{\i}}$};
      \foreach \i in {-2.5,-1.5,1.5, 2.5}
      \draw [line width=2pt, opacity=0.1] (\i, 0.5) -- (\i, -0.5);
      \foreach \i in {-0.5,0.5}
      \draw [line width=2pt, opacity=0.1] (\i, 1.5) -- (\i, -1.5);
    \end{tikzpicture}
    \caption*{Example for $D_4$}
  \end{figure}
\end{example}

A sequence $\X^{(n)}$ of random variables is \Dfn{asymptotically normal} if its normalized version converges in distribution towards the standard Gaussian distribution.
In symbols,
\[
\frac{\X^{(n)}-\Exp{\X^{(n)}}}{\Var{\X^{(n)}}^{1/2}} \overset{\mathcal D}{\longrightarrow} \cN(0,1).
\]

The first main theorem considers a further generalization of $d$-descents.
An \Dfn{antichain} in the root poset is a set of positive roots that are pairwise incomparable.

\begin{theorem}
  \label{thm:antichain}
  Let $\{\Phi^{(n)}\}_{n \geq 1}$ be a sequence of finite crystallographic root systems and let $\{\Psi^{(n)}\}_{n \geq 1}$ be a sequence of antichains in its root poset of $\Phi^{(n)}$ of increasing cardinality, $|\Psi^{(n)}|~\longrightarrow~\infty$.
  The corresponding random variable $\X_{\Psi^{(n)}}$  is then asymptotically normal.
  The rate of convergence in distribution is bounded by $|\Psi^{(n)}|^{-1/2}$.
\end{theorem}

Since the set $\Phiddes{d}$ is an antichain for any~$d$, we obtain a central limit theorem for $d$-descents as a direct consequence.

\begin{corollary}
  \label{cor:gendesc}
  Let $\{\Phi^{(n)}\}_{n \geq 1}$ be a sequence of finite crystallographic root systems and let $\{d_n\}_{n \geq 1}$  be a sequence of integers such that $|\Phiddes{d_n}| \longrightarrow \infty$.
  The corresponding random variable $\X_{\Phiddes{d_n}}$  is then asymptotically normal with rate of convergence in distribution being bounded by~$|\Phiddes{d_n}|^{-1/2}$.
\end{corollary}

\begin{remark}
  The proof of Theorem~\ref{thm:antichain} generalizes verbatim to any subset of roots that may be send to simple roots by an element of~$W$.
  If we allow any such subset of roots in a (not necessarily crystallographic) root system, we obtain asymptotic normality if $\Var{\X_{\Psi^{(n)}}} \rightarrow \infty$ as given in~\cite[Theorem~6.2]{Kahle_2019}.
\end{remark}

We next consider the analogous result for $d$-inversions.

\begin{theorem}
  \label{thm:geninv}
  Let $\{\Phi^{(n)}\}_{n \geq 1}$ be a sequence of finite crystallographic root systems and let $\{d_n\}_{n \geq 1}$  be a sequence of integers such that $|\Phidinv{d_n}| \longrightarrow \infty$.
  The corresponding random variable $\X_{\Phidinv{d_n}}$  is then asymptotically normal.
\end{theorem}

The following corollary finally provides rates of convergence where our used method is applicable.
This generalizes the known rates for the symmetric group discussed at the end of Section~\ref{sec:intro}.
In the corollary, we assume without loss of generality that for all $n$, there exists a root in $\Phi^{(n)}$ of height $d_n$ and refer to Section~\ref{sec:cltdinv} for details concerning this assumption.

\begin{corollary}
  \label{cor:geninv}
  Let $\Phi^{(n)}$ be of rank~$r_n$ and let $\{\Phi_i\}$ denote its irreducible components.
  Set
  \begin{itemize}
    \item $r_A$ to be the rank of the union of those components $\Phi_i$ for which $\rk(\Phi_i) < d_n$,
    \item $r_B$ to be the rank of the union of those components $\Phi_i$ for which $d_n \leq \rk(\Phi_i) \leq d_n^2$, and
    \item $r_C$ to be the rank of the union of those components $\Phi_i$ for which $d_n^2 < \rk(\Phi_i)$,
  \end{itemize}
  (and we observe that $r_n = r_A + r_B + r_C$).
  We then have the following bounds on the rates of convergence:
  \begin{itemize}
    \item if $r_A$ grows linearly in $r_n$, we have the bound $r_n^{-1/2}$,
    \item if $r_B$ grows linearly in $r_n$ and $d_n$ grows faster than $r_n^{2/3}$, we have the bound $r_n \cdot d_n^{-3/2}$,
    \item if $r_C$ grows linearly in $r_n$ and $d_n$ grows slower than $r_n^{1/3}$, we have the bound $r_n^{-1/2}\cdot d_n^{3/2}$.
  \end{itemize}
\end{corollary}

In particular, this corollary shows that if $\Phi^{(n)}$ is irreducible (of then necessarily classical type), we obtain rates of convergence in the same situations as for the symmetric group.

\medskip

We finish this section with a discussion of the used techniques in the proofs and with an open problem.

\medskip

Based on a well-known Coxeter-theoretic result, the asymptotic normality in Theorem~\ref{thm:antichain} can be deduced in a rather straightforward way from the knowledge about asymptotic normality for descents, see Remark~\ref{rem:alternativeproof}.
We provide a more detailed alternative proof using the dependency graph method which also provides the proposed rate of convergence.
Corollary~\ref{cor:gendesc} then directly follows.
In addition, we provide explicit formulas for their variances in the classical types in Theorems~\ref{thm:VarA},\ref{thm:VarB},\ref{thm:VarC}, and \ref{thm:VarD}.

\medskip

Unfortunately, we do not have a uniform consideration for $d$-inversions.
Based on provided explicit formulas for their variances in the classical types in Theorems~\ref{thm:VarA}, \ref{thm:VarB}, \ref{thm:VarC}, and \ref{thm:VarD}, we prove a general lower bound in Lemma~\ref{lem:varlowerbound} which is then used via the dependency graph method to prove the proposed asymptotic normality in Theorem~\ref{thm:geninv} and rates of convergence in Corollary~\ref{cor:geninv}.

\medskip

The situation of $d$-inversions could be considered as a special case of \Dfn{order ideals} in the root poset, \ie, subsets $\Psi \subseteq \Phiplus$ such that $\beta \leq \gamma \in \Psi$ implies $\beta \in \Psi$.
We neither found an example of a sequence of order ideals of increasing size that does not have a central limit, nor do we have any particular reason to expect them to always have a central limit.
We are therefore not aware of any natural subset generalizing $d$-inversions that may be treated uniformly in the present context.
We propose this as an open question.
\begin{center}
  \emph{Do sequences of order ideals of increasing cardinality always have a central limit?}
\end{center}

\section{Random inversions for finite root systems}
\label{sec:randominvs}

The aim of this section is to understand how random variables $\{ \X_\beta \mid \beta \in \Phiplus\}$ interact.
The two main result are Theorems~\ref{thm:Cov} and \ref{thm:CovSet}.

\begin{theorem}
  \label{thm:Cov}
  For two positive roots $\beta, \gamma \in \Phiplus$, we have
  \[
  \Cov(\X_{\beta}, \X_{\gamma}) = \pm \left(\frac{1}{4} -  \frac{1}{2 \cdot \ord(\beta,\gamma)}\right),
  \]
  with the sign being positive if $\langle \beta, \gamma \rangle \geq 0$ and negative if $\langle \beta, \gamma \rangle \leq 0$.
  In particular, $\X_\beta$ and $\X_\gamma$ are independent if and only if~$\beta$ and $\gamma$ are orthogonal.
\end{theorem}

\begin{remark}
  One may rephrase this statement into an expression avoiding the case distinction by casting it in terms of the angle between~$\beta$ and $\gamma$.
  Let $\beta,\gamma \in \Phiplus$ and let
  \[
  \varphi = \angle(\beta,\gamma) = \arccos\big(\frac{\langle \beta,\gamma\rangle}{|\beta|\cdot|\gamma|}\big)
  \]
  be the angle in between them.
  Then
  \[
  \Cov(\X_{\beta}, \X_{\gamma}) = \tfrac{1}{12\pi}(3\pi-6\varphi)\,.
  \]
\end{remark}

\begin{example}
  \label{ex:dihedralrootsystems}
  The positive roots in the irreducible crystallographic root systems of dihedral types $A_2, B_2, C_2$ and $G_2$ are
  
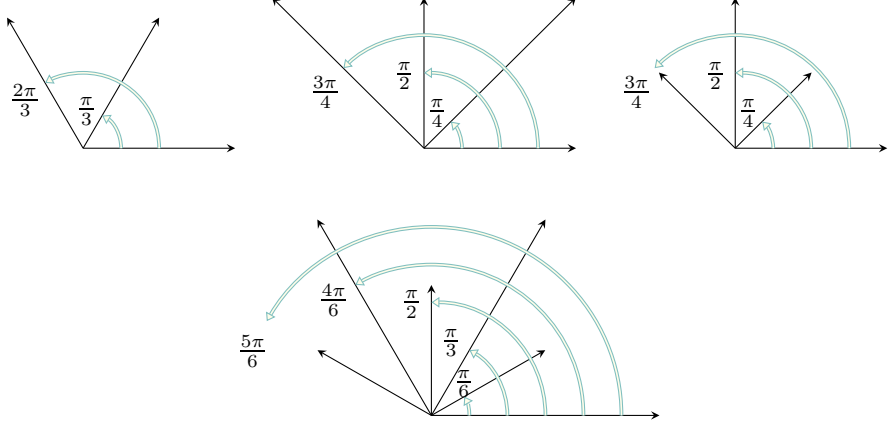
\begin{center}
    \begin{tikzpicture}
      \foreach \i/\j in {0/2, 60/2, 120/2}{
        \draw[-stealth](0,0) -- (  \i:\j);
      }
      \foreach \i/\j in {60/0.5, 120/1}{
        \pic at (\j,0) {arc arrow=0: \i:\j cm};
      }
      \foreach \i/\j/\l in {85/0.5/$\frac{\pi}{3}$, 140/1/$\frac{2\pi}{3}$}{
        \node at ( \i:\j) {\l};
      }
    \end{tikzpicture}
    \quad
    \begin{tikzpicture}
      \foreach \i/\j in {0/2, 45/2.83, 90/2, 135/2.83}{
        \draw[-stealth](0,0) -- (  \i:\j);
      }
      \foreach \i/\j in {45/0.5, 90/1, 135/1.5}{
        \pic at (\j,0) {arc arrow=0: \i:\j cm};
      }
      \foreach \i/\j/\l in {70/0.5/$\frac{\pi}{4}$, 105/1/$\frac{\pi}{2}$, 150/1.5/$\frac{3\pi}{4}$}{
        \node at ( \i:\j) {\l};
      }
    \end{tikzpicture}
    \quad
    \begin{tikzpicture}
      \foreach \i/\j in {0/2, 45/1.42, 90/2, 135/1.42}{
        \draw[-stealth](0,0) -- (  \i:\j);
      }
      \foreach \i/\j in {45/0.5, 90/1, 135/1.5}{
        \pic at (\j,0) {arc arrow=0: \i:\j cm};
      }
      \foreach \i/\j/\l in {70/0.5/$\frac{\pi}{4}$, 105/1/$\frac{\pi}{2}$, 150/1.5/$\frac{3\pi}{4}$}{
        \node at ( \i:\j) {\l};
      }
    \end{tikzpicture}
    
    \vspace{25pt}
    
    \begin{tikzpicture}
      \foreach \i/\j in {0/3, 30/1.73, 60/3, 90/1.73, 120/3, 150/1.73}{
        \draw[-stealth](0,0) -- (  \i:\j);
      }
      \foreach \i/\j in {30/0.5, 60/1, 90/1.5, 120/2.0, 150/2.5}{
        \pic at (\j,0) {arc arrow=0: \i:\j cm};
      }
      \foreach \i/\j/\l in {45/0.6/$\frac{\pi}{6}$, 75/1/$\frac{\pi}{3}$, 100/1.5/$\frac{\pi}{2}$, 130/2/$\frac{4\pi}{6}$, 160/2.5/$\frac{5\pi}{6}$}{
        \node at ( \i:\j) {\l};
      }
    \end{tikzpicture}
  \end{center}
\end{example}

\begin{theorem}
\label{thm:CovSet}
  Let $\Psi,\Psi' \subseteq \Phiplus$.
  If~$\Psi$ and~$\Psi'$ are orthogonal, \ie, $\ord(\beta,\gamma) = 2$ for all $\beta\in\Psi$ and $\gamma\in\Psi'$, then the sets $\{\X_\beta\}_{\beta\in\Psi}$ and $\{\X_\gamma\}_{\gamma\in\Psi'}$ of random variables are independent.
\end{theorem}

The proofs of the two theorems are very similar.
We start with proving Theorem~\ref{thm:Cov} by partitioning the Weyl group~$W$ depending on the behaviour on the positive roots $\beta,\gamma \in \Phiplus$.
For given signs $\varepsilon, \delta \in \{+, -\}$, we define the set
\[
\W{\varepsilon}{\delta}{\beta}{\gamma} = \big\{w \in W \mid \varepsilon\cdot w(\beta) \in \Phi^+, \delta\cdot w(\gamma) \in \Phi^+\big\} \subseteq W
\]
as those elements in the Weyl group that sends the given roots~$\beta$ and~$\gamma$ to positive or, respectively, negative roots depending on the given signs $\varepsilon,\delta$.
The core of the proof is the following relationship between the cardinalities of these four sets.

\begin{proposition}
  \label{prop:Wpartition}
  Let $\beta, \gamma \in \Phi^+$.
  Then
  \begin{align*}
    \big|W^{(++)}(\beta, \gamma)\big| &= \big|W^{(--)}(\beta, \gamma)\big|, \\
    \big|W^{(+-)}(\beta, \gamma)\big| &= \big|W^{(-+)}(\beta, \gamma)\big|. \\
    \intertext{If moreover $\beta \neq \gamma$, then}
    \big|W^{(++)}(\beta, \gamma)\big| &= \big(\ord(\beta,\gamma)-1\big)\cdot\big|W^{(+-)}(\beta,\gamma)\big| \text{ if } \langle \beta,\gamma \rangle \geq 0, \\
    \big|W^{(+-)}(\beta, \gamma)\big| &= \big(\ord(\beta,\gamma)-1\big)\cdot\big|W^{(++)}(\beta,\gamma)\big| \text{ if } \langle \beta,\gamma \rangle \leq 0.
  \end{align*}
  In particular, all four sets have equal cardinality if and only if~$\beta$ and $\gamma$ are orthogonal.
\end{proposition}

We first show that it is enough to restrict attention to the two-dimensional subspace $\lspan\{\beta, \gamma\}$ of~$V$ and the parabolic subgroup of~$W$ acting on this subspace.
One considers, for a set $\Gamma \subseteq \Delta$ of simple roots, the \Dfn{parabolic subgroup}
\[
W_\Gamma = \langle s_\alpha \mid \alpha \in \Gamma \rangle \leq W
\]
with corresponding root system $\Phi^+_\Gamma = \Phi^+ \cap \, \lspan(\Gamma)$ where the latter is the linear subspace of~$V$ generated by~$\Gamma$.
It is well-known that the parabolic subgroup~$W_\Gamma$ comes with a \Dfn{parabolic quotient} $W^\Gamma$ given by
\[
W^\Gamma = \{ w \in W \mid w(\alpha) \in \Phi^+ \text{ for all } \alpha \in \Gamma \}.
\]
Every element $w \in W$ then has a unique decomposition $w = w^\Gamma\cdot w_\Gamma$ with $w^\Gamma \in W^\Gamma$ and $w_\Gamma \in W_\Gamma$.

\begin{lemma}
  \label{lem:parabolicpositivity}
  Let~$\Gamma \subseteq \Delta$ and let $w \in W$ with parabolic decomposition $w = w^\Gamma \cdot w_\Gamma$.
  For any $\beta \in \Phi^+_\Gamma$, it then holds that
  \[
  w(\beta) \in \Phi^+ \Longleftrightarrow w_\Gamma(\beta) \in \Phi^+.
  \]
\end{lemma}

\begin{proof}
  Since every element in $\Phi^+_\Gamma$ is a non-negative linear combination of elements in~$\Gamma$, the definition of the parabolic quotient immediately implies that $w^\Gamma(\Phi^+_\Gamma) \subseteq \Phi^+$.
  Thus, if $w_\Gamma(\beta) \in \Phi^+$ then it is also in $\Phi^+_\Gamma$ and therefore $w(\beta) = w^\Gamma ( w_\Gamma(\beta))$ is also in~$\Phi^+$.
  If on the other hand, $w_\Gamma(\beta) \notin \Phi^+$, then $w_\Gamma(\beta) \in -\Phi^+_\Gamma$.
  Applying the previous argument to $-w_\Gamma(\beta)$ then shows that $w(\beta) \in -\Phi^+ = \Phi^-$.
\end{proof}

Looking at the definition of the random variable $\X_\beta$ in~\eqref{eq:Xbetadef}, this lemma shows that~ $\X_\beta$ only depends on the smallest parabolic root subsystem containing~$\beta$.

\begin{corollary}
  \label{cor:Xparabolic}
  Let $\Gamma \subseteq \Delta$ with parabolic subgroup $W_\Gamma$ and let $\beta\in\Phi^+_\Gamma$.
  The random variable $\X_\beta$ is the same when considered on~$W$ and when considered on~$W_\Gamma$.
\end{corollary}

\begin{proof}
  Since $\X_\beta$ is a Bernoulli random variable, we check that
  \begin{align*}
    \Pro{\X_\beta = 1} &= \frac{\big|\{w \in W \mid w(\beta) \in \Phi^-\}\big|}{|W|} = \frac{|W^\Gamma|\! \cdot\! \big|\{w \in W_\Gamma \mid w(\beta) \in \Phi^-\}\big|}{|W^\Gamma|\cdot|W_\Gamma|} \\
    &= \frac{\big|\{w \in W_\Gamma \mid w(\beta) \in \Phi^-\}\big|}{|W_\Gamma|}\qedhere
  \end{align*}
\end{proof}

\begin{proof}[Proof of Proposition~\ref{prop:Wpartition}]
  The first two properties are implied by the observation that the map $w \mapsto w_\circ\cdot w$ is a bijection on~$W$ interchanging inversions and non-inversions,
  \[
  \Inv(w_\circ\cdot w) = \Phiplus \setminus \Inv(w)\,.
  \]
  Let now be $\lspan\{\beta, \gamma\}$ the linear subspace of~$V$ spanned by the two positive roots~$\beta, \gamma \in \Phi^+$.
  Then it is classical that there exists an element $w \in W$ and simple roots $\Gamma = \{\alpha,\alpha'\} \subseteq \Delta$ such that
  \[
  w\big(\lspan\{\beta,\gamma\}\big) = \lspan(\Gamma).
  \]
  So in particular, $w(\beta),w(\gamma) \in \Phi^+_\Gamma$.
  Since we are only interested in the cardinalities of the four sets and multiplying by~$w$ is a bijection on~$W$, we may assume without loss of generality that $w = e$ and $\beta, \gamma \in \Phi^+_\Gamma$.
  It then holds by Lemma~\ref{lem:parabolicpositivity} that
  \[
  W^{(\varepsilon, \delta)}(\beta, \gamma) = W^\Gamma \times W^{(\varepsilon,\delta)}_\Gamma(\beta,\gamma).
  \]
  This shows that one only needs to check the properties
  \begin{align*}
    \big|W^{(++)}(\beta, \gamma)\big| &= (\ord{(s_\beta}{s_\gamma)}-1)\cdot\big|W^{(+-)}(\beta,\gamma)\big| \text{ if } \langle \beta,\gamma \rangle \geq 0, \\
    \big|W^{(+-)}(\beta, \gamma)\big| &= (\ord{(s_\beta}{s_\gamma)}-1)\cdot\big|W^{(++)}(\beta,\gamma)\big| \text{ if } \langle \beta,\gamma \rangle \leq 0
  \end{align*}
  for the dihedral types $A_1 \times A_1, A_2, B_2, C_2, G_2$ to conclude the proposition.
  This check is an easy exercise, we consider the case of type~$G_2$ in the following example, the others are similar.
\end{proof}

\begin{example}
  Consider the Weyl group~$W$ of type~$G_2$ with simple reflections $\sref=\{s,t\}$.
  The set of reflections is then
  \[
  \rref = \{s, sts, ststs, tstst, tst, t\}
  \]
  which we denote by $s_{\beta_1},\dots,s_{\beta_6}$ in this order with corresponding positive roots $\Phi^+ = \{\beta_1,\dots,\beta_6\}$.
  Inspecting Example~\ref{ex:dihedralrootsystems}, we see that
  \[
  \ord(\beta_i,\beta_j) =
  \begin{cases}
    6 & |i-j| \in \{1,5\},\\
    3 & |i-j| \in \{2,4\},\\
    2 & |i-j|= 3.
  \end{cases}
  \]
  Using
  \[
  \begin{array}{rclcrcl}
    \Inv(e)     &=& \{\}                                        && \Inv(ststst) &=& \{\beta_1,\beta_2,\beta_3,\beta_4,\beta_5,\beta_6\} \\
    \Inv(s)     &=& \{\beta_1\}                                 && \Inv(tstst)  &=& \{\beta_2,\beta_3,\beta_4,\beta_5,\beta_6\} \\
    \Inv(st)    &=& \{\beta_1,\beta_2\}                         && \Inv(tsts)   &=& \{\beta_3,\beta_4,\beta_5,\beta_6\} \\
    \Inv(sts)   &=& \{\beta_1,\beta_2,\beta_3\}                 && \Inv(tst)    &=& \{\beta_4,\beta_5,\beta_6\} \\
    \Inv(stst)  &=& \{\beta_1,\beta_2,\beta_3,\beta_4\}         && \Inv(ts)     &=& \{\beta_5,\beta_6\} \\
    \Inv(ststs) &=& \{\beta_1,\beta_2,\beta_3,\beta_4,\beta_5\} && \Inv(t)      &=& \{\beta_6\} \\
  \end{array}
  \]
  we obtain for example for $\beta_2$ and $\beta_3$ the decomposition
  \begin{align*}
    W^{(++)} &= \{e, s, tst, ts, t\}\,, \\
    W^{(--)} &= \{sts,stst,ststs,ststst=tststs,tstst\}\,, \\
    W^{(+-)} &= \{tsts\}\,, \\
    W^{(-+)} &= \{st\}\,.
  \end{align*}
  We finally observe with $\langle \beta_2, \beta_3\rangle > 0$ that
  \[
  |W^{(++)}| = |W^{(--)}|, \quad |W^{(+-)}| = |W^{(-+)}|, \quad |W^{(++)}| = (\underbrace{\ord(\beta_2\beta_3)}_{=6}-1)\cdot |W^{(+-)}| \,.
  \]
\end{example}

\begin{proof}[Proof of Theorem~\ref{thm:Cov}]
  We calculate
  \begin{align*}
    \Cov(\X_\beta , \X_\gamma)\ &= \Exp{\X_\beta \cdot \X_\gamma} - \Exp{\X_\beta}\cdot \Exp{\X_\gamma} = \Pro{\X_\beta = \X_\gamma=1} - \frac{1}{4}\\ 
    &= \frac{|W^{(--)}(\beta,\gamma)|}{|W|} - \frac{1}{4}\,.
  \end{align*}
  By the first two equalities in Proposition~\ref{prop:Wpartition}, this may be rewritten to
  \[
  \Cov(\X_\beta , \X_\gamma) = \frac{|W^{(--)}(\beta,\gamma)|}{2\cdot \big(|W^{(--)}(\beta,\gamma)|+|W^{(+-)}(\beta,\gamma)|\big)} - \frac{1}{4}
  \]
  If $\beta = \gamma$, we have $W^{(+-)}(\beta,\gamma) = \emptyset$ and obtain $\Cov(\X_\beta,\X_\gamma) = \tfrac{1}{2} - \tfrac{1}{4} = \tfrac{1}{4}$.
  
  \medskip
  
  Let now $\beta \neq \gamma$ and assume $\langle \beta, \gamma \rangle \leq 0$.
  The forth equality in Proposition~\ref{prop:Wpartition} then gives
  \begin{align*}
    \Cov(\X_\beta , \X_\gamma) &= \frac{|W^{(--)}(\beta,\gamma)|}{2\cdot (|W^{(--)}(\beta,\gamma)|+(\ord(\beta,\gamma)-1)\cdot|W^{(--)}(\beta, \gamma)|)} - \frac{1}{4}\\
    &= \frac{1}{2\cdot \ord(\beta,\gamma) } - \frac{1}{4}.
  \end{align*}
  The case of $\langle \beta, \gamma \rangle \geq 0 $ is analogous.
\end{proof}

\begin{proof}[Proof of Theorem~\ref{thm:CovSet}]
  Let $\Psi,\Psi' \subseteq \Phiplus$ be orthogonal sets of roots as given in the assumption.
  As in the proof of Corollary~\ref{cor:Xparabolic}, there exists an element $w \in W$ and simple roots~$\Gamma \subseteq \Delta$ such that
  \[
    w\big(\lspan(\Psi)\big) = \lspan(\Gamma)\,.
  \]
  Since the inner product on~$V$ is $W$-invariant, we moreover have that $w(\Psi')$ is contained in the orthogonal complement of~$\lspan(\Gamma)$.
  Since $W = w\cdot W$, we may assume without loss of generality that~$w = e \in W$ is the identity element.

  We use the unique (set) decomposition $W = W^\Gamma \cdot W_\Gamma$ into a parabolic quotient and the parabolic subgroup generated by~$\Gamma$ to conclude the proof.
  To this end, let $\X = \{\X_\beta\}_{\beta \in \Psi}$ and $\X' = \{\X_\gamma\}_{\gamma \in \Psi'}$.
  We aim to show
  \[
    \Pro{\X' = \delta} = \Pro{\X' = \delta\ |\ \X = \varepsilon }
  \]
  for any fixed outcomes $\varepsilon \in \{ 0,1 \}^{\Psi}, \delta \in \{ 0,1 \}^{\Psi'}$.
  This follows from the observation that the outcome $\X' = \delta$ only depends on the parabolic quotient~$W^\Gamma$ since $W_\Gamma$ pointwise fixes~$\Psi'$, while the outcome $\X = \varepsilon$ only depends on the parabolic subgroup~$W_\Gamma$ by Lemma~\ref{lem:parabolicpositivity}.
  We make this argument precise in the following calculations.
  To simplify notation, we write for $w \in W$ that
  \[
    \varepsilon\cdot w(\Psi) \in \Phi^-\quad \text{ if for all } \beta\in\Psi\text{ we have }
    \begin{cases}
      w(\beta) \in \Phi^-   \text{ if } \varepsilon(\beta) = 1\,, \\
      w(\beta) \in \Phiplus \text{ if } \varepsilon(\beta) = 0\,,
    \end{cases}
  \]
  and we analogously write $\delta\cdot w(\Psi') \in \Phi^-$.
  We then have
  \begin{align*}
    \Pro{\X' = \delta\ |\ \X = \varepsilon }
    &= \frac{\big|\{w \in W \mid \varepsilon\cdot w(\Psi) \in \Phi^-,\ \delta\cdot w(\Psi') \in \Phi^- \}\big|}{\big|\{w \in W \mid \varepsilon\cdot w(\Psi) \in \Phi^- \}\big|} \\[5pt]
    &= \frac{\big|\{w^\Gamma \in W^\Gamma, w_\Gamma \in W_\Gamma \mid \varepsilon\cdot w_\Gamma(\Psi) \in \Phi^-,\ \delta\cdot w^\Gamma(\Psi') \in \Phi^- \}\big|}{|W^\Gamma|\cdot|\{ w_\Gamma \in W_\Gamma \mid \varepsilon\cdot w_\Gamma(\Psi) \in \Phi^-|} \\[5pt]
    &= \frac{\big|\{w \in W^\Gamma \mid \delta\cdot w(\Psi') \in \Phi^- \}\big|}{|W^\Gamma|} \\
    &= \frac{\big|\{w \in W \mid \delta\cdot w(\Psi') \in \Phi^- \}\big|}{|W|} \\
    &= \Pro{\X' = \delta}\,. \qedhere
  \end{align*}
\end{proof}

\section{Central limit theorems}
\label{section:CLT}

Equipped with the necessary Coxeter-theoretic properties from Section~\ref{sec:randominvs}, we use the dependency graph method to prove the claimed central limit theorems, together with bounds on rates of convergence where the used method is applicable.
The computations for $d$-inversions also depend on the concrete calculations of the variances for $d$-inversions given in Section~\ref{sec:concretevariances}.

\subsection{The dependency graph method}
\label{sec:depgraph}

Let $\{ \X^{(n,j)}\}$ with $n \geq 0$ and $1 \leq j \leq k_n$ be a triangular array of random variables.
For fixed~$n$, define the graph~$G_n$ to have vertices $\{1,\dots,k_n\}$ and edges $\{i,j\}$ with $i \neq j$ if $\X^{(n,i)}$ and~$\X^{(n,j)}$ are dependent.
It is called the \Dfn{$n$-th dependency graph} of this triangular array if for any two disjoined subsets $I,J \subseteq \{1,\dots,k_n\}$ without edges between~$I$ and~$J$ in~$G_n$, it holds that also the corresponding sets of random variables are independent.
In this case, we say that the triangular array \Dfn{admits dependency graphs} of \Dfn{dependency degree} $(\delta_n)_{n \geq 0}$ where $\delta_n$ is the maximal vertex degree of~$G_n$.

\medskip

We present the following theorem due to Janson~\cite{janson_1988} in a version tailor made for our situation.

\begin{theorem}[\!\!{\cite[Theorem~2]{janson_1988}}]
  \label{thm:DepGraph}
  Let $\{\X^{(n,j)}\}$ with $n \geq 0$ and $1 \leq j \leq k_n$ be a triangular array of Bernoulli random variables, admitting dependency graphs of dependency degree~$(\delta_n)_{n\geq 0}$.
  Then the sequence $\X^{(n)} = \X^{(n,1)} + \dots + \X^{(n,k_n)}$ is asymptotically normal if there exists an $m \in \NN$ such that
  \[
  \frac{k_n \cdot \delta_n^{m-1}}{\Var{\X^{(n)}}^{m/2}} \longrightarrow 0\,.
  \]
\end{theorem}

\begin{corollary}[\!\!{\cite[Corollary 4.3 \& Theorem 4.7]{Feray_2019}}]
  \label{cor:rateofconvergence}
  If the assumption in the previous theorem is satisfied for $m = 3$, we obtain a bound for the rate of convergence given by
  \[
  \sup_{x \in \RR}\Big| \Pro{\X^{(n)} \leq x} - \Pro{\cN(0,1) \leq x} \Big| \leq k_n \cdot \delta_n^2 \cdot \Var{\X^{(n)}}^{-3/2} \longrightarrow 0\,.
  \]
\end{corollary}

We use Theorem~\ref{thm:DepGraph} to prove asymptotic normality for antichains in Theorem~\ref{thm:antichain}, directly implying the main result for $d$-descents in Corollary~\ref{cor:gendesc}.
Using then the concrete variances for $d$-inversions, we prove the second main result for $d$-inversions stated in Theorem~\ref{thm:geninv}.

\medskip

We close this section with introducing three notions for later usage.
For functions $f,g: \NN_{+} \to \RR_{\ge 0}$, we write $f \lesssim g$, if there exists $\varepsilon > 0$ and an~$N \in \NN$ such that for all $n \geq N$, it holds that $f(n) \leq \varepsilon g(n)$. 
We moreover write $f \ll g $, if for all $\varepsilon > 0$ there exists an~$N \in \NN$ with this property. Observe that $f \ll g \Leftrightarrow f(n)/g(n) \longrightarrow 0\,.$
We finally write
\begin{center}
  $f \approx g$ for $f \lesssim g \lesssim f$.
\end{center}

\subsection{Central limit theorems for antichains in root posets and $d$-descents}

Let $\{\Phi^{(n)}\}_{n \geq 1}$ be a sequence of finite crystallographic root systems and let $\{\Psi^{(n)}\}_{n \geq 1}$ be a sequence of antichains in the respective root posets.
We consider the random variable $\X_{\Psi^{(n)}}$ as given in~\eqref{def:RV} and give two proofs of its asymptotic normality.
We use the dependency graph method above and then remark that one may also instead rely on the result for descents in~\cite{Kahle_2019}.
The dependency graph method yields in addition the proposed bound on the rate of convergence.

\medskip

We start with collecting the following properties.

\begin{lemma}
  \label{lem:sommers}
  Let $\Psi \subset \Phiplus$ be an antichain.
  Then
  \begin{itemize}
    \item $\langle \beta, \gamma\rangle \leq 0$ for all $\beta \neq \gamma$ in $\Psi$,
    \item there are at most $|\Psi|-1$ many pairs of dependent random variables $\X_\beta,\X_\gamma$ for $\beta,\gamma \in \Psi$, and
    \item the dependency degree $\delta$ is at most~$3$.
  \end{itemize}
\end{lemma}

\begin{proof}
  All properties are clearly satisfied if $\Psi \subset \Delta$ is a set of simple roots:
  Any two simple roots have inner product less than or equal to zero and the dependency graph of $\X_\Psi$ for $\Psi \subset \Delta$ equals the Coxeter graph of the parabolic subsystem generated by~$\Psi$.
  Since these properties are also invariant under the group action of~$W$ on~$\Phi$, the conclusions follow from~\cite{Sommers_2005} where it was shown that any antichain in the root poset may be send to simple roots by an element in~$W$.
\end{proof}

\begin{proposition}
  \label{prop:VarCard}
  The variance of $\X_{\Psi^{(n)}}$  grows linearly with the cardinality of the antichain, \ie,
  \[
  \Var{\X_{\Psi^{(n)}}} \approx |\Psi^{(n)}|.
  \]
\end{proposition}

\begin{proof}
  Together with Lemma~\ref{lem:sommers}, it follows from Theorem~\ref{thm:Cov} that we may rewrite $\Var{\X_{\Psi^{(n)}}}$ as
  \[
    \Var{\X_{\Psi^{(n)}}} = \sum \limits_{\beta \in \Psi^{(n)}} \Var{\X_\beta} + \sum\limits_{\substack{\beta, \gamma \in \Psi^{(n)} \\ \beta \neq \gamma}} \Cov(\X_\beta, \X_\gamma)
    =  \frac{1}{4} \left(|\Psi^{(n)}| - \sum\limits_{\substack{\beta, \gamma \in \Psi^{(n)} \\ \beta \neq \gamma}} \Big( 1 - \frac{2}{\ord({\beta},{\gamma})}\Big)\right).
  \]
  With $\ord(\beta,\gamma) \in \{2,3,4,6\}$, we immediately obtain $\Var{\X_{\Psi^{(n)}}} \leq \frac{|\Psi^{(n)}|}{4}$.
  We aim to also show the lower bound $\Var{\X_{\Psi^{(n)}}} \geq \frac{|\Psi^{(n)}|}{12}$ by showing that
  \[
    \sum\limits_{\substack{\beta, \gamma \in \Psi^{(n)} \\ \beta \neq \gamma}} \Big( 1 - \frac{2}{\ord({\beta},{\gamma})}\Big) \leq \frac{2}{3}|\Psi^{(n)}|\,.
  \]
  It follows, as in the proof of Lemma~\ref{lem:sommers}, from~\cite{Sommers_2005} that $\Psi^{(n)}$ are the simple roots of a parabolic subsystem.
  We may thus decompose $\Psi^{(n)} = \Delta_1 \cup \dots \cup \Delta_k$ into its irreducible components, meaning that $\Delta_i$ contains the simple roots of on irreducible root system.
  We therefore have $\ord(\beta,\gamma) = 2$ for~$\beta \in \Delta_i$ and~$\gamma \in \Delta_j$ with $i \neq j$, and also that for each~$\Delta_i$, there are exactly $2|\Delta_i|-2$ many pairs $(\beta,\gamma) \in \Delta_i^2$ with $\ord(\beta,\gamma) \geq 3$ and for at most two of them, say $(\beta,\gamma)$ and $(\gamma,\beta)$, we have $\ord(\beta,\gamma) = \ord(\gamma,\beta) > 3$.
  We then obtain
  \[
    \sum\limits_{\substack{\beta, \gamma \in \Psi^{(n)} \\ \beta \neq \gamma}} \Big( 1 - \frac{2}{\ord({\beta},{\gamma})}\Big)
    = \sum_{i=1}^k \sum\limits_{\substack{\beta, \gamma \in \Delta_i \\ \beta \neq \gamma}} \Big( 1 - \frac{2}{\ord({\beta},{\gamma})}\Big)
    \leq \sum_{i=1}^k \left( \frac{2}{3}\big(|\Delta_i|-2\big)+\frac{4}{3}\right) = \frac{2}{3}|\Psi^{(n)}|\,,
  \]
  as desired.
\end{proof}

\begin{proof}[Proof of Theorem~\ref{thm:antichain}]
  The random variable $\X_{\Psi^{(n)}}$ is a sum of Bernoulli random variables.
  We have seen in Theorem~\ref{thm:CovSet} that this is a triangular array of random variables admits dependency graphs.
  Lemma~\ref{lem:sommers} showed that its dependency degree $\delta_n$ is globally bounded by~$3$.
  We have moreover seen in Proposition~\ref{prop:VarCard} that the variance of $\X_{\Psi^{(n)}}$ grows linearly with $k_n = |\Psi^{(n)}|$.
  We therefore obtain that
  \[
  \frac{k_n \cdot \delta_n^{m-1}}{\Var{\X_{\Psi^{(n)}}}^{m/2}} \approx \big|\Psi^{(n)}\big|^{\tfrac{2-m}{2}} \longrightarrow 0
  \]
  for any $m \geq 3$.
  $\X_{\Psi^{(n)}}$ is thus asymptotically normal by Theorem~\ref{thm:DepGraph} and the bound on the rate of convergence is $|\Psi^{(n)}|^{-1/2}$ as given by Corollary~\ref{cor:rateofconvergence}.
\end{proof}

Without the rate of convergence, we may as well more easily reduce the situation of Theorem~\ref{thm:antichain} to previously known results.

\begin{remark}
  \label{rem:alternativeproof}
  One may as well derive the asymptotic normality as claimed in Theorem~\ref{thm:antichain} from the asymptotic normality in~\cite{Kahle_2019} as follows.
  By~\cite{Sommers_2005}, any antichain in the root poset may be send to simple roots by an element in~$W$.
  We thus reduce the situation of $\X_\Psi^{(n)}$ to the case that $\Psi^{(n)} \subseteq \Delta^{(n)}$ for the simple roots $\Delta^{(n)}$ of the root system $\Phi^{(n)}$.
  We have moreover seen in Corollary~\ref{cor:Xparabolic} that $\X_\Psi^{(n)}$ may equivalently be considered inside $W_{\Delta^{(n)}}$, and this case is treated in~\cite[Theorem~6.2]{Kahle_2019}.
  The approach used here uses different techniques and in particular yields the claimed rate of convergence.
\end{remark}

\subsection{Central limit theorems for $d$-inversions}
\label{sec:cltdinv}

Let $\{\Phi^{(n)}\}_{n \geq 1}$ be a sequence of finite crystallographic root systems and let $\{d_n\}_{n\geq 1}$ a sequence of integer.
We consider the random variable
\[
\X_{\Phidinv{d_n}} = \sum_{\beta \in \Phidinv{d_n}} \X_\beta\quad\text{with}\quad \Phidinv{d_n} = \{\beta \in \Phiplus \mid \height{\beta}\leq d_n\}
\]
and prove its asymptotic normality using the dependency graph method.
We in addition provide bounds on the rate of convergence where the method is applicable.
Based on the concrete calculations in Section~\ref{sec:concretevariances}, we start with the following lemma giving lower bounds for the variance depending on the situation.

\begin{lemma}
  \label{lem:varlowerbound}
  There exists a global constant $\varepsilon > 0$ such that for any irreducible root system $\Phi$ of rank~$r$ and any parameter~$d$, we have
  \[
  \Var{\X_{\Phidinv{d+1}}} \geq \Var{\X_{\Phidinv{d}}} >  \begin{cases}
    \varepsilon\cdot r^3      &\text{if } \hspace{6pt}  r \leq d\,,\\[5pt]
    \varepsilon\cdot d^3      &\text{if } \hspace{5pt} d \leq r \leq d^2\,, \\[5pt]
    \varepsilon\cdot r\cdot d &\text{if } \hspace{5pt} d^2 \leq r \,.
  \end{cases}
  \]
\end{lemma}
\begin{proof}
  The first property that the variance (weakly) increases in~$d$ can be checked by observing that the derivative in~$d$ of the involved polynomials---see Section~\ref{sec:concretevariances}---is non-negative, or by concretely calculating the difference between two variances for consecutive parameters in~$d$.
  Since these calculations are not enlightening, we leave the details to the reader.
  
  \medskip
  
  The inequality $\Var{\X_{\Phidinv{d+1}}} \geq \Var{\X_{\Phidinv{d}}}$ may be used to obtain
  \[
  \Var{\X_{\Phidinv{d}}} \geq \Var{\X_{\Phidinv{k}}}
  \]
  where $k = \min\{ d, r/2\}$.
  We therefore only need to treat the concrete variances for $2d \leq r$.
  
  The first case is $r \leq d$ and we obtain
  \[
  \Var{\X_{\Phidinv{d}}} \geq \Var{\X_{\Phidinv{r/2}}} > \varepsilon \cdot r^3 \, .
  \]
  The second case $d \leq r \leq d^2$ is split into two subcases.
  If $d \leq r/2$, we directly have
  \[
  \Var{\X_{\Phidinv{d}}} > \varepsilon \cdot d^3
  \]
  and if $d \geq r/2$, we have
  \[
  \Var{\X_{\Phidinv{d}}} \geq \Var{\X_{\Phidinv{r/2}}} > \varepsilon \cdot r^3 \geq \varepsilon \cdot d^3\,.
  \]
  In the final case $d^2 \leq r$, we have $2d \leq d^2 \leq r$ (the case $d=1$ also fulfills the bound trivially).
  Thus,
  \[
  \Var{\X_{\Phidinv{d}}} > \max{(rd, d^3)} = rd \, .
  \]
  where we used that in all classical types, there are terms~$rd$ and~$d^3$ with positive coefficients in the polynomial expressions for $2d \leq r$.
  All involved parameters~$\varepsilon$ may a priori differ, but---after also taking the finite number of exceptional groups into account---we may choose one global such~$\varepsilon$.
\end{proof}

We decompose a given root system $\Phi = \Phi^{(n)}$ depending on the parameter $d = d_n$ into its irreducible components according to Lemma~\ref{lem:varlowerbound} as
\[
\Phi = \Phi_A \cup \Phi_B \cup \Phi_C = \bigcup_{i \in A}\Phi_i\ \cup\ \bigcup_{i \in B} \Phi_i\ \cup\ \bigcup_{i \in C} \Phi_i
\]
where we have that different irreducible components $\Phi_i$ and $\Phi_j$ for $i \neq j$ live in orthogonal subspaces.
We moreover denote their respective ranks $r_i = \rk(\Phi_i)$ and the indices are decomposed as
\[
A = \{ i \mid r_i < d \}, \quad B = \{ i \mid d \leq r_i \leq d^2\},\quad C = \{ i \mid d^2 < r_i\} \,.
\]
We set $r_A = \sum_{i \in A} r_i = \rk(\Phi_A)$ and $r_B$ and $r_C$ analogously, and observe that $r_n = \rk(\Phi^{(n)}) = r_A + r_B + r_C$.

\medskip

In the considered setup, we assume $|\Phidinv{d_n}| \longrightarrow \infty$.
This implies that also $r_n = \rk(\Phi^{(n)}) \longrightarrow \infty$.
We moreover assume without loss of generality that there exists positive roots for $\Phi^{(n)}$ of height~$d_n$ in the root poset.
Otherwise $\Phidinv{d_n} = \Phidinv{d_n-1}$ and we may replace~$d_n$ by $d_n-1$ without modifying the considered random variable.
The root poset of $\Phi^{(n)}$ consists of the disjoint union of the root posets (in the sense that their Hasse diagrams are disjoint) of its irreducible components $\{ \Phi^{(n)}_i \mid i \in A \cup B \cup C\}$.
We finally denote by $\delta_n$ the $n$-th dependency degree.

\begin{proposition}
  \label{prop:order of kn deltan}
  We have
  \[
  |\Phidinv{d_n}| \lesssim r_n\cdot d_n\quad\text{and}\quad\delta_n \lesssim d_n \,.
  \]
\end{proposition}

\begin{proof}
  The first property follows from the observation that the number of elements of a given height in the root poset is bounded by the rank.
  
  For the second, let $\beta \in \Phidinv{d_n}$.
  We first discuss the situation that~$\beta$ is in an irreducible component of one of the classical types~$A$--$D$.
  In these cases, we have seen that~$\beta$ is of one of the respective forms $e_i - e_j$, $e_i$, $e_i+e_j$ and a positive root~$\gamma$ is not orthogonal to~$\beta$ if and only if~$\gamma$ lives in the same irreducible component (so $\gamma$ is of one the respective forms $e_k -e_\ell$, $e_k$, $e_k+e_\ell$) and~$\beta$ and $\gamma$ share one of the used indices.
  This implies that the number of roots not orthogonal to~$\beta$ is at most $4d_n$.
  On the other hand, if~$\beta$ is in an irreducible component of exceptional type, then the number of roots not orthogonal to~$\beta$ is globally bounded.
\end{proof}

\begin{proof}[Proof of Theorem~\ref{thm:geninv} \& Corollary~\ref{cor:geninv}]
  We write $\sigma_n^2 = \Var{\X_{\Phidinv{d_n}}}$ and analogously $\sigma_i^2$ for the variance of the corresponding random variable of the irreducible component $\Phi_i$ and decompose
  \[
  \sigma_n^2 = \sigma_A^2 + \sigma_B^2 + \sigma_C^2 = \sum_{i\in A}\sigma_i^2 + \sum_{i\in B}\sigma_i^2 + \sum_{i\in C}\sigma_i^2
  \]
  according to the decomposition $\Phi^{(n)} = \Phi_A \cup \Phi_B \cup \Phi_C$ where we used that the corresponding random variables of the irreducible components are independent because of the orthogonality between roots in different irreducible components, compare Theorem~\ref{thm:Cov}.
  By Lemma~\ref{lem:varlowerbound}, we moreover have a global constant~$\varepsilon$ with
  \[
  \sigma_A^2 = \sum_{i\in A}\sigma_i^2 > \varepsilon\cdot r_A^3,\quad
  \sigma_B^2 = \sum_{i\in B}\sigma_i^2 > \varepsilon\cdot |B|\cdot d^3,\quad
  \sigma_C^2 = \sum_{i\in C}\sigma_i^2 > \varepsilon\cdot d\cdot r_C\,.
  \]
  For the first inequality, we have used $\sigma_i^2 > \varepsilon' \sum_{i \in A}r_i^3 \approx (\sum_{i \in A} r_i)^3 = r_A^3$.
  The second and third inequalities are immediate from the lemma.
  
  \medskip
  
  We aim to use Theorem~\ref{thm:DepGraph} for $m = 5$ in general and for $m = 3$ where possible, we again emphasize that the dependency graph method is applicable by Theorem~\ref{thm:CovSet}.
  Proposition~\ref{prop:order of kn deltan} gives
  \[
  \frac{|\Phidinv{d_n}| \cdot \delta_n^{m-1}}{\sigma_n^m} \lesssim \frac{r_n\cdot d_n^m}{\sigma_n^m}\,.
  \]
  We have already seen that $r_n = r_A + r_B + r_C \longrightarrow \infty$.
  We thus consider the following three cases:
  
  \smallskip
  
  \noindent $r_A \approx r_n$:
  In this case, we have $r_n^3 \lesssim \sigma_n^2$ and thus
  \[
  \frac{r_n\cdot d_n^m}{\sigma_n^m} \lesssim r_n^{1-\tfrac{3m}{2}} \cdot d_n^m \lesssim r_n^{1-\tfrac{m}{2}} \longrightarrow 0 \text{ for } m=3\,.
  \]
  
  \smallskip
  
  \noindent $r_C \approx r_n$:
  In this case, we have $r_n \cdot d_n \lesssim \sigma_n^2$ and thus
  \[
  \frac{r_n\cdot d_n^m}{\sigma_n^m} \lesssim r_n^{1-\tfrac{m}{2}} \cdot d_n^{\tfrac{m}{2}} \lesssim r_n^{1-\tfrac{m}{4}} \longrightarrow 0 \text{ for } m=5\,.
  \]
  The latter approximation comes from the observation that $d_n^2 < r_i$ for any $i \in C$ and thus $d_n^2 \lesssim r_n$.
  Indeed, if we even assume $d_n^3 \ll r_n$, we obtain
  \[
  r_n^{1-\tfrac{m}{2}} \cdot d_n^{\tfrac{m}{2}} \ll r_n^{1 - \tfrac{m}{2} + \tfrac{m}{6}} = r_n^{1 - \tfrac{2m}{6}}\,,
  \]
  and thus
  \[
  \frac{r_n\cdot d_n^m}{\sigma_n^m} \lesssim r_n^{1-\tfrac{m}{2}} \cdot d_n^{\tfrac{m}{2}} \longrightarrow 0 \text{ for } m = 3
  \]
  in this case.
  
  \smallskip
  
  \noindent $r_B \approx r_n$:
  In this case, we have $|B|\cdot d_n^3 \lesssim \sigma_n^2$ and thus
  \[
  \frac{r_n\cdot d_n^m}{\sigma_n^m} \lesssim r_n \cdot d_n^{-\tfrac{m}{2}} \big|B\,\big|^{-\tfrac{m}{2}}\,.
  \]
  We may now compute
  \begin{equation}
    r_n \approx r_B = \sum_{i\in B}r_i \leq \sum_{i \in B} d_n^2 = d_n^2\cdot|B| \label{eq:goestoinfty}
  \end{equation}
  where we used that $r_i \leq d_n^2$ for all $i \in B$.
  This gives
  \[
  r_n \cdot d_n^{-\tfrac{m}{2}} \big|B\,\big|^{-\tfrac{m}{2}} \lesssim d_n^{2-\tfrac{m}{2}} \big|B\,\big|^{1-\tfrac{m}{2}} \longrightarrow 0 \text{ for } m = 5\,,
  \]
  using~\eqref{eq:goestoinfty} to show that either $d_n \rightarrow \infty$ or $|B|\rightarrow\infty$.
  If we even assume $r_n^2 \ll d_n^3$, we obtain
  \[
  r_n \cdot d_n^{-\tfrac{m}{2}} \big|B\,\big|^{-\tfrac{m}{2}} \longrightarrow 0\text{ for } m=3\,. \qedhere
  \]
\end{proof}

\appendix

\section{Concrete variances}
\label{sec:concretevariances}

Based on Theorem~\ref{thm:Cov}, we provide concrete formulas for the variances of $d$-descents and $d$-inversions in all irreducible types~$A,B,C$ and~$D$.
The concrete formulas for the variances of $d$-descents are not used in this paper and are only provided for its own sake.
The variances for $d$-inversions on the other hand are then used in Section~\ref{section:CLT} to prove Lemma~\ref{lem:varlowerbound} which is then further used to derive asymptotic normality as claimed in Theorem~\ref{thm:geninv} and Corollary~\ref{cor:geninv}.

\subsection{Type $A_{n-1}$}\label{subsection:typeA}

In this section, we provide the variances for $d$-descents and $d$-inversions of type $A_{n-1}$.
The case of $d$-inversions were already considered in \cite[Lemmata~1\&2]{bona_2008} and \cite[Theorem 1]{Pike_2011}.

\begin{theorem}
  \label{thm:VarA}
  Let $\X_\Psi$ as in~\eqref{def:RV} and $\Phiddes{d}, \Phidinv{d}$ as in~\eqref{eq:ddesdinv}.
  We then have for $d \leq n$ that
  \begin{align*}
    \Var{\X_{\Phiddes{d}}} &= \left\{\begin{array}{lr}
      \frac{1}{12} (n+d) & \text{\hspace*{246pt}if } 2d \leq n \\[1em]
      \frac{1}{4} (n-d) & \text{if } 2d \geq n
    \end{array}\right.
    \\[1em]
    \Var{\X_{\Phidinv{d}}} &= \left\{\begin{array}{lr}
      \frac{1}{18} d^3 + \frac{1}{24} d^2 + ( \frac{1}{12} n - \frac{1}{72}) d & \text{if } 2d \leq n\\[1em]
      - \frac{1}{6} d^3 + (\frac{1}{3} n -\frac{7}{24}) d^2 + ( -\frac{1}{6} n^2 + \frac{5}{12}n - \frac{1}{8}) d + (\frac{1}{36}n^3 - \frac{1}{12}n^2 + \frac{1}{18}n) & \text{if } 2d \geq n\\
    \end{array}\right.
  \end{align*}
\end{theorem}

The proof is provided in several steps.
We first partition the positive roots $\Phiplus = N_1 \cup N_2 \cup \dots \cup N_{n-1}$ according to their height,
\[
N_a = \{ \beta \in \Phiplus \mid \height(\beta) = a \} = \big\{\Nroot{ij} \mid 1 \leq i < j = i+a \leq n\big\}\,.
\]
The following lemma is then immediate from this definition.

\begin{lemma}\label{lem:cardNa}
  We have
  \[
  N_a \neq \emptyset \iff 1 \leq a < n, \quad |N_a| = n - a \text{ for } 1 \leq a < n\,.
  \]
\end{lemma}

According to Theorem~\ref{thm:Cov}, the random variables $\X_{\Nroot{ij}}, \X_{\Nroot{k\ell}}$ are independent if and only if $\{i,j\} \cap \{k,\ell\} = \emptyset$.
Moreover, we have for $k \neq j$ the covariances
\begin{align}
  \Cov(\X_{\Nroot{ij}},\X_{\Nroot{ik}}) = \Cov(\X_{\Nroot{ji}},\X_{\Nroot{ki}}) = \tfrac{1}{12}, \quad \Cov(\X_{\Nroot{ji}},\X_{\Nroot{ik}}) = \Cov(\X_{\Nroot{ij}},\X_{\Nroot{ki}}) = -\tfrac{1}{12}\,.
  \label{eq:paircovariances}
\end{align}
We therefore need to compute the following four cardinalities.

\begin{lemma}
  \label{lem:cardN}
  For $1 \leq a,b < n$, we have
  \begin{align*}
    \big|\{\Nroot{ij},\Nroot{k\ell} \in N_a \times N_b \mid i=k, j \neq k\} \big| &= \big|\{\Nroot{ij},\Nroot{k\ell} \in N_a \times N_b \mid i \neq k, j=\ell \}\big| \\
    &= \left\{\begin{array}{ll}
      n- \max\{a,b\} & a \neq b\\
      0 & a = b
    \end{array}\right. \\[1em]
    \big|\{\Nroot{ij},\Nroot{k\ell} \in N_a \times N_b \mid i = \ell\}\big| &= \big|\{\Nroot{ij},\Nroot{k\ell} \in N_a \times N_b \mid j=k\}\big| \\
    &= n - \min\{n, a+b\}
  \end{align*}
\end{lemma}

\begin{proof}
  First, we observe that the two sets in the first equality are empty for $a=b$.
  For all other cases, one uses
  \[
  N_c = \big\{ \Nroot{i,i+c} \mid 0 < i \leq n-c \big\} = \big\{ \Nroot{j-c,j} \mid c < j \leq n \big\}
  \]
  both for $N_a$ and $N_b$ to obtain the given counting formulas.
\end{proof}

\begin{proof}[Proof of Theorem~\ref{thm:VarA}]
  Combining Lemma~\ref{lem:cardN} with~\eqref{eq:paircovariances}, we obtain from the case $a=b=d$ that
  \[
  \Var{\X_{\Phiddes{d}}} = \tfrac{1}{4} |N_d| - \tfrac{2}{12} \big(n- \min(n, 2d)\big)\,.
  \]
  This is equivalent to the two cases in the formula for $\Var{\X_{\Phiddes{d}}}$.
  On the other hand, we have
  \[
  \Var{\X_{\Phidinv{d}}} = \tfrac{1}{4}\sum\limits_{1\leq a \leq d} |N_a| +  \tfrac{2}{12} \sum\limits_{1\leq a \neq b \leq d} \big(n- \max(a,b) \big) - \tfrac{2}{12} \sum\limits_{1\leq a, b \leq d} \big(n-\min(n,a+b) \big)\,.
  \]
  The first sum comes from the terms $\Var{\X_{\Nroot{ij}}}$ for all $\Nroot{ij} \in \Phidinv{d}$.
  The second sum comes from the terms of the form $\Cov(\X_{\Nroot{ij}},\X_{\Nroot{ik}})$.
  It can be rewritten as
  \[
  \sum\limits_{1\leq a \neq b \leq d} (n- \max(a,b) ) = 2 \cdot \sum\limits_{1\leq a < b \leq d} (n-b) \\
  = 2 \cdot \Big(- \tfrac{1}{6} ( d-1)d(2d-3n+2)\Big)\,.
  \]
  For the third sum, first note that the summands are zero for $a+b \geq n$.
  We may thus rewrite this as
  \[
  \sum\limits_{1\leq a, b \leq d} (n-\min(n,a+b)) = \sum\limits_{\substack{1 \leq a,b \leq d \\ a+b \leq n}} (n-(a+b)) = \sum\limits_{1 \leq a\leq d} \, \sum\limits_{1 \leq b\leq \min(n-a, d)} (n-(a+b))\,.
  \]
  Assume first $2d \leq n$.
  Since $a,b \leq d$ and thus $d \leq n-a$, we obtain
  \[
  \sum\limits_{1\leq a, b \leq d} (n-\min(n,a+b)) = \sum\limits_{1 \leq a\leq d} \sum\limits_{1 \leq b\leq d} (n-(a+b)) = d^2(-d+n-1).
  \]
  Assume next $2d \geq n$.
  We further consider two cases $d \leq n-a$ and $d > n-a$.
  \begin{align*}
    \sum\limits_{1\leq a, b \leq d} (n-\min(n,a+b)) &= \sum\limits_{1 \leq a\leq n-d} \sum\limits_{1 \leq b\leq d} (n-(a+b)) + \sum\limits_{n-d < a\leq d} \sum\limits_{1 \leq b\leq n-a} (n-(a+b))\\
    &= - \tfrac{1}{2} d (n-2)(d-n) + \tfrac{1}{6} (2d-n)(d^2-dn+n^2-3n+2)
  \end{align*}
  Combining the three sums under consideration, one obtains the stated variance.
\end{proof}

\subsection{Type $B_n$}

In this section, we provide the variances for $d$-descents and $d$-inversions of type $B_{n}$.

\begin{theorem}
  \label{thm:VarB}
  Let $\X_\Psi$ as in~\eqref{def:RV} and $\Phiddes{d}, \Phidinv{d}$ as in~\eqref{eq:ddesdinv}.
  We then have for ${1 \leq d \leq 2n-1}$ that
  \begin{align*}
    \Var{\X_{\Phiddes{d}}} &= \\
    &\left\{\begin{array}{lr}
      \frac{1}{24} d + \frac{1}{12} n  + \frac{1}{12}  &\text{ \hspace*{229pt} if } d \leq \frac{n}{2} \text{, } d \text{ even} \\[1em]
      \frac{1}{24} d + \frac{1}{12} n  + \frac{1}{24}  & \text{if } d \leq \frac{n}{2} \text{, } d \text{ odd} \\[2em]
      \frac{1}{24} d + \frac{1}{12} n  + \frac{1}{6}  & \text{if } \frac{n}{2} < d \leq \frac{2}{3}n \text{, } d \text{ even} \\[1em]
      \frac{1}{24} d + \frac{1}{12} n  + \frac{1}{8}  & \text{if } \frac{n}{2} < d \leq \frac{2}{3}n \text{, } d \text{ odd} \\[2em]
      \frac{1}{24} d + \frac{1}{12} n  & \text{if } \frac{2}{3} n < d \leq n \text{, } d \text{ even} \\[1em]
      \frac{1}{24} d + \frac{1}{12} n  + \frac{1}{8}  & \text{if } \frac{2}{3} n \leq d \leq n \text{, } d \text{ odd} \\[2em]
      -\frac{1}{8} d + \frac{1}{4} n  & \text{if } n \leq d \text{, } d \text{ even} \\[1em]
      -\frac{1}{8} d + \frac{1}{4} n  + \frac{1}{8}  & \text{if } n \leq d \text{, } d \text{ odd}
    \end{array}\right.
    \\[1em]
    \Var{\X_{\Phidinv{d}}} &= \\
    &\left\{\begin{array}{lr}
      \frac{1}{36} d^3 + \frac{1}{48} d^2 + (\frac{1}{12} n + \frac{1}{72}) d & \text{if } d \leq \frac{n}{2} \text{, } d \text{ even} \\[1em]
      \frac{1}{36} d^3 + \frac{1}{48} d^2 + (\frac{1}{12} n + \frac{1}{72}) d + \frac{1}{48} & \text{if } d \leq \frac{n}{2} \text{, } d \text{ odd} \\[2em]
      \frac{1}{36} d^3 + \frac{3}{16} d^2 + (-\frac{1}{12} n + \frac{7}{72}) d + (\frac{1}{24} n^2 - \frac{1}{24}n) & \text{if } \frac{n}{2} \leq d \leq \frac{2}{3}n \text{, } d \text{ even} \\[1em]
      \frac{1}{36} d^3 + \frac{3}{16} d^2 + (-\frac{1}{12} n + \frac{7}{72}) d+ (\frac{1}{24} n^2 - \frac{1}{24} n + \frac{1}{48}) & \text{if } \frac{n}{2} \leq d < \frac{2}{3}n \text{, } d \text{ odd} \\[2em]
      \frac{1}{36} d^3 + (\frac{1}{6} n - \frac{1}{36}) d + ( - \frac{1}{24}n^2 + \frac{1}{24} n ) & \text{if } \frac{2}{3} n \leq d \leq n \text{, } d \text{ even} \\[1em]
      \frac{1}{36} d^3 + (\frac{1}{6} n + \frac{7}{72}) d + (- \frac{1}{24} n^2 - \frac{1}{24} n + \frac{1}{24}) & \text{if } \frac{2}{3} n \leq d \leq n \text{, } d \text{ odd} \\[2em]
      -\frac{1}{12} d^3 + (\frac{1}{3} n - \frac{1}{24}) d^2 + (-\frac{1}{3} n^2+ \frac{1}{6} n - \frac{1}{24}) d +  (\frac{1}{9} n^3 + \frac{1}{18} n) & \text{if } n \leq d  \text{, } d \text{ even} \\[1em]
      -\frac{1}{12} d^3 + (\frac{1}{3} n - \frac{1}{24}) d^2 + (-\frac{1}{3} n^2+ \frac{1}{6} n + \frac{1}{12}) d +  (\frac{1}{9} n^3 - \frac{1}{36} n + \frac{1}{24}) & \text{if } n \leq d  \text{, } d \text{ odd} \\[1em]
    \end{array}\right.
  \end{align*}
\end{theorem}

The proof is provided in several steps.
We first partition the positive roots according to their height,
\begin{align*}
  N_a  &= \big\{\Nroot{ij} \mid 1 \leq i < j \leq n \mid j-i\! =\! a \big\},\\
  O_a  &= \big\{\Oroot{i} \mid 1 \leq i \leq n \mid i\! =\! a \big\},\\
  P_a  &= \big\{\Proot{ij} \mid 1 \leq i < j \leq n \mid j+i\! =\! a\big\}\,.
\end{align*}
so that the roots of height~$a$ are $N_a \cup O_a \cup P_a$.
Analogous to Lemma~\ref{lem:cardNa}, we have the following cardinalities.

\begin{lemma}
  We have
  \begin{itemize}
    \item $O_a \neq \emptyset \Leftrightarrow 1 \leq a \leq n$ and $|O_a| = 1$ for $1 \leq a \leq n$
    \item $P_a \neq \emptyset \Leftrightarrow 3 \leq a \leq 2n-1$ and $|P_a| =  \left\{\begin{array}{lr}
      \lfloor\frac{a-1}{2}\rfloor & \text{if } 3 \leq a \leq n \\[1em]
      \lfloor \frac{2n-a+1}{2} \rfloor & \text{if } n+1 \leq a \leq 2n-1
    \end{array}\right.$
  \end{itemize}
\end{lemma}

\begin{proof}
  This follows directly from the definition.
\end{proof}
The situation is slightly more delicate than in type $A_{n-1}$.
We thus extract the covariances between roots of the three form into a proposition.
To this end, we introduce the sets
\[
N_{\leq d} = \bigcup_{a=1}^d N_a, \quad O_{\leq d} = \bigcup_{a=1}^d O_a,\quad P_{\leq d} = \bigcup_{a=1}^d P_a\,.
\]

\begin{proposition}
  \label{prop:covariancesB}
  For $1 \leq d \leq 2n-1$ we have,
  \begin{align*}
    2 \cdot &\Cov(\Pinv, \Ninv) =\\
    &\left\{\begin{array}{lr}
      -\frac{1}{18}d^3 + \tfrac{1 }{16} d^2 + \tfrac{7}{72}d  & \text{if } d \leq \frac{n}{2} \text{, d even} \\
      -\frac{1}{18}d^3 + \tfrac{1 }{16} d^2 + \tfrac{1}{18}d - \tfrac{1}{16}& \text{if } d \leq \frac{n}{2} \text{, d odd} \\ [1em]
      \frac{1}{6} d^3 + (-\tfrac{1}{3} n + \tfrac{1}{16}) d^2 + (\tfrac{1}{6}n^2 + \tfrac{1}{24})d + (- \tfrac{1}{36} n^3 + \tfrac{1}{36}n)  & \text{if } \tfrac{n}{2} \leq d \leq \frac{2n}{3} \text{, d even} \\
      \frac{1}{6} d^3 + (-\tfrac{1}{3} n + \tfrac{1}{16}) d^2 + (\tfrac{1}{6}n^2)d + (- \tfrac{1}{36} n^3 + \tfrac{1}{36}n - \tfrac{1}{16})  & \text{if } \tfrac{n}{2} \leq d \leq \frac{2n}{3}  \text{, d odd} \\ [1em]
      \frac{1}{6} d^3 + (-\tfrac{1}{3} n - \tfrac{1}{8}) d^2 + (\tfrac{1}{6}n^2 + \tfrac{1}{4}n - \tfrac{1}{12})d + (- \tfrac{1}{36} n^3 - \tfrac{1}{12}n^2 + \tfrac{1}{9}n)  & \text{if } \tfrac{2n}{3} \leq d \leq n  \text{, d even} \\
      \frac{1}{6} d^3 + (-\tfrac{1}{3} n - \tfrac{1}{8}) d^2 + (\tfrac{1}{6}n^2 + \tfrac{1}{4}n)d + (- \tfrac{1}{36} n^3 - \tfrac{1}{12}n^2 + \tfrac{1}{36}n - \tfrac{1}{24})& \text{if } \tfrac{2n}{3} \leq d \leq n \text{, d odd} \\ [1em]
      -\frac{1}{18} d^3 + (\tfrac{1}{4} n + \tfrac{1}{24}) d^2 + (-\tfrac{1}{3}n^2 - \tfrac{1}{6} n - \tfrac{1}{36})d + ( \tfrac{1}{9} n^3 + \tfrac{1}{6}n^2 + \tfrac{1}{18}n) & \text{if } n \leq d \text{, d even} \\
      -\frac{1}{18} d^3 + (\tfrac{1}{4} n + \tfrac{1}{24}) d^2 + (-\tfrac{1}{3}n^2 - \tfrac{1}{6} n + \tfrac{1}{18})d + ( \tfrac{1}{9} n^3 - \tfrac{1}{6}n^2 - \tfrac{1}{36}n - \tfrac{1}{24}) & \text{if } n \leq d \text{, d odd} \\ [1em]
    \end{array}\right.
    \\[1em]
    2 \cdot &\Cov(\Ninv, \Oinv) = \\
    &\left\{\begin{array}{lr}
      -\frac{1}{8}d^2 - \tfrac{1 }{8} d & \text{\hspace*{220pt} if } d \leq \frac{n}{2}\\
      \frac{3}{8}d^2 + (-\tfrac{1 }{2} n + \tfrac{1}{8}) d + (\tfrac{1}{8}n^2 - \tfrac{1}{8} n)& \text{if } \frac{n}{2} \leq d \leq n
    \end{array}\right.
    \\[1em]
    2 \cdot &\Cov(\Pinv, \Oinv) = \\
    & \left\{\begin{array}{lr}
      \tfrac{1}{8}d^2 - \tfrac{1}{4} d & \text{\hspace*{185pt} if } d \leq n \text{, } d \text{ even}\\
      \tfrac{1}{8}d^2 - \tfrac{1}{4} d + \tfrac{1}{8}& \text{if } d \leq n \text{, } d \text{ odd}\\ [1em]
      -\tfrac{1}{8}d^2 + \tfrac{1}{2} n d + (-\tfrac{1}{4} n^2 - \tfrac{1}{4}n) & \text{if } n \leq d \text{, } d \text{ even}\\
      -\tfrac{1}{8}d^2 + \tfrac{1}{2} n d + (-\tfrac{1}{4} n^2 - \tfrac{1}{4}n + \tfrac{1}{8}) & \text{if } n \leq d \text{, } d \text{ odd}
    \end{array}\right.
    \\[1em]
    &\Cov(\Pinv, \Pinv) = \\
    &\left\{\begin{array}{lr}
      \tfrac{1}{36}d^3 - \tfrac{1}{12} d^2 + \tfrac{1}{18} d & \text{\hspace*{35pt}if } d \leq n \text{, } d \text{ even}\\
      \tfrac{1}{36}d^3 - \tfrac{1}{12} d^2 + \tfrac{7}{72} d - \tfrac{1}{24}& \text{if } d \leq n \text{, } d \text{ odd}\\ [1em]
      -\tfrac{1}{36}d^3 + (\tfrac{1}{12} n + \tfrac{1}{24}) d^2 + (-\tfrac{1}{6} n - \tfrac{1}{72}) d + (- \tfrac{1}{36} n^3 + \tfrac{1}{24} n^2 + \tfrac{5}{72}n)  & \text{if } n \leq d \text{, } d \text{ even}\\
      -\tfrac{1}{36}d^3 + (\tfrac{1}{12} n + \tfrac{1}{24}) d^2 + (-\tfrac{1}{6} n - \tfrac{1}{36}) d + (- \tfrac{1}{36} n^3 + \tfrac{1}{24} n^2 + \tfrac{5}{72}n - \tfrac{1}{24}) & \text{if } n \leq d \text{, } d \text{ odd}
    \end{array}\right.
    \\[1em]
    &\Cov(\Oinv, \Oinv) = \begin{array}{lr} \tfrac{1}{4}d  & \text{\hspace*{274pt}if } d \leq n  \end{array}
  \end{align*}
\end{proposition}

According to Theorem~\ref{thm:Cov} there are several combinations of roots in $N_a, P_b, O_c$ with non-zero covariance to investigate.
We have
\begin{align*}
  N \sim N: &\qquad \Cov(\X_{\Nroot{ij}},\X_{\Nroot{ik}}) = \tfrac{1}{12} \text{ for } j \neq k &&\Cov(\X_{\Nroot{ij}},\X_{\Nroot{ki}}) = -\tfrac{1}{12} \text{ for } j \neq k\\[5pt]
  &\qquad \Cov(\X_{\Nroot{ij}},\X_{\Nroot{jk}}) = -\tfrac{1}{12} \text{ for } i \neq k &&\Cov(\X_{\Nroot{ij}},\X_{\Nroot{kj}}) = \tfrac{1}{12} \text{ for } i \neq k\\[5pt]
  N \sim O: &\qquad \Cov(\X_{\Nroot{ij}},\X_{\Oroot{i}}) = -\tfrac{1}{8},  &&\Cov(\X_{\Nroot{ij}},\X_{\Oroot{j}}) = \tfrac{1}{8} \\[5pt]
  N \sim P: &\qquad \Cov(\X_{\Nroot{ij}},\X_{\Proot{ik}}) = -\tfrac{1}{12} \text{ for } j \neq k   &&\Cov(\X_{\Nroot{ij}},\X_{\Proot{ki}}) = -\tfrac{1}{12} \text{ for } j \neq k  \\[5pt]
  &\qquad \Cov(\X_{\Nroot{ij}},\X_{\Proot{jk}}) = \tfrac{1}{12} \text{ for } i \neq k  &&\Cov(\X_{\Nroot{ij}},\X_{\Proot{kj}}) = \tfrac{1}{12} \text{ for } i \neq k  \\[5pt]
  P \sim O: &\qquad \Cov(\X_{\Proot{ij}},\X_{\Oroot{i}}) = \tfrac{1}{8},  &&\Cov(\X_{\Proot{ij}},\X_{\Oroot{j}}) = \tfrac{1}{8} \\[5pt]
  P \sim P: &\qquad \Cov(\X_{\Proot{ij}},\X_{\Proot{ik}}) = \tfrac{1}{12}, \quad \text{ if } j \neq k &&\Cov(\X_{\Proot{ij}},\X_{\Proot{ki}}) = \tfrac{1}{12} \quad \text{ if } j \neq k\\[5pt]
  &\qquad \Cov(\X_{\Proot{ij}},\X_{\Proot{jk}}) = \tfrac{1}{12}, \quad \text{ if } i \neq k &&\Cov(\X_{\Proot{ij}},\X_{\Proot{kj}}) = \tfrac{1}{12} \quad \text{ if } i \neq k
\end{align*}

We therefore compute case-by-case the cardinalities of those interactions.

\begin{proposition}
  \label{prop:interactionNP}
  We consider the two sets $N_a$ and $P_b$ for type $B_n$ with $1 \leq a \leq n$ and $3 \leq b \leq 2n-1$.
  Depending on satisfied inequalities of the parameters $n,a$ and $b$, the cardinalities of the four non-zero variance cases for $N_a \sim P_b$ above are as follows:
  
  \bigskip
  
  \begin{align*}
    \Big|\{\Nroot{ij},\Proot{kl} \in N_a \times P_b \mid i=k \}\Big| &=\qquad
    \begin{tabular}{c ||c|c}
      & \tiny{$n-a \leq \lfloor\tfrac{b-1}{2}\rfloor$} & \tiny{$\lfloor\tfrac{b-1}{2}\rfloor \leq n-a$} \\
      \hline \hline
      \tiny{$1 \leq b-n$} & $2n-(a+b)+1$ & $\lfloor\tfrac{1}{2} (2n-b+1) \rfloor$ \\ \hline
      \tiny{$b-n \leq 1$} & $n-a$ & $\lfloor\tfrac{b-1}{2}\rfloor$
    \end{tabular}\\[1em]
    \Big|\{\Nroot{ij},\Proot{kl} \in N_a \times P_b \mid i=\ell \}\Big| &=\qquad
    \begin{tabular}{c ||c|c}
      & \tiny{$b-1 \leq n-a$} & \tiny{$n-a\leq b-1$} \\
      \hline \hline
      \tiny{$\lceil \tfrac{b+1}{2} \rceil$} & $\lfloor \tfrac{b-1}{2} \rfloor$ & $\lfloor\tfrac{1}{2} (2n-2a-b+1) \rfloor$
    \end{tabular}\\[1em]
    \Big|\{\Nroot{ij},\Proot{kl} \in N_a \times P_b \mid j=k \}\Big| &=\qquad
    \begin{tabular}{c ||c}
      & \tiny{$\lfloor \tfrac{b-1}{2}\rfloor$} \\
      \hline \hline
      \tiny{$a+1 \leq b-n$} & $\lfloor\tfrac{1}{2} (2n-b+1)\rfloor$ \\ \hline
      \tiny{$b-n \leq a+1$} & $\lfloor\tfrac{1}{2} (b-2a-1)\rfloor$
    \end{tabular}\\[1em]
    \Big|\{\Nroot{ij},\Proot{kl} \in N_a \times P_b \mid j=\ell \}\Big| &=\qquad
    \begin{tabular}{c ||c|c}
      & \tiny{$n \leq b-1$} & \tiny{$b-1\leq n$} \\
      \hline \hline
      \tiny{$a+1 \leq \lceil \tfrac{b-1}{2} \rceil$} & $\lceil\tfrac{1}{2} (2n - b+1)\rceil$ & $\lfloor\tfrac{b+1}{2}\rfloor$ \\ \hline
      \tiny{$\lceil \tfrac{b-1}{2} \rceil \leq a+1$} & $n-a$ & $b-a-1$
    \end{tabular}\\[1em]
  \end{align*}
\end{proposition}

\begin{proof}
  We give a proof for the first counting formula, the arguments for the others are analogous.
  The set $N_a$ contains exactly one pair $(i,i+a)$ for each $1 \leq i \leq n-a$ and $P_b$ contains exaclty one pair $(k,b-k)$ for each $\max(1, b-n) \leq k \leq \lfloor \tfrac{b-1}{2} \rfloor$.
  Combining these, we obtain that the cardinality for those pairs with $i=k$ equals
  \[
  \max(1, b-n) \leq i=k \leq \min (n-a, \lfloor \tfrac{b-1}{2} \rfloor).
  \]
  This leads to the four cases and the corresponding cardinalities.
\end{proof}

Given these cardinalities we can give a proof for the stated covariances between the roots of~$N_a$ and~$P_b$.

\begin{proof}[Proof of Proposition~\ref{prop:covariancesB}]
  We will give a proof for the first case, the others follow analogously tedious considerations.
  Suppose $a,b \leq d \leq \frac{n}{2}$ and~$d$ even.
  Investigating the first table under the given restrictions, the only relevant case is $b-n \leq 1, \lfloor \tfrac{b-1}{2} \rfloor \leq n-a$.
  We thus obtain
  \begin{align*}
    \sum\limits_{1 \leq a,b \leq d} \Big|\{\Nroot{ij},\Proot{k\ell} \in N_a \times P_b \mid i=k \}\Big|  = \sum\limits_{a=1}^d \sum\limits_{b=1}^d \lfloor \tfrac{b-1}{2} \rfloor = \frac{1}{4} d^2 (d-2)\,.
  \end{align*}
  Similarly, the investigation of the second table yields the only relevant case $b-1 \leq n-a$, leading to
  \begin{align*}
    \sum\limits_{1 \leq a,b \leq d} \Big|\{\Nroot{ij},\Proot{k\ell} \in N_a \times P_b \mid i=\ell \}\Big|  = \sum\limits_{a=1}^d \sum\limits_{b=1}^d \lfloor \tfrac{b-1}{2} \rfloor = \frac{1}{4} d^2 (d-2)\,.
  \end{align*}
  Investigating next the third table, the only relevant case is $b-n \leq a+1$.
  Since the restrictions on $j = k$ for are $\lfloor \tfrac{b-1}{2} \rfloor$ for the upper bound and $a+1$ for the lower bound, we obtain $b \geq 2(a+1)+1$ and thus
  \begin{align*}
    \sum\limits_{1 \leq a,b \leq d} \Big|\{\Nroot{ij},\Proot{k\ell} \in N_a \times P_b \mid j=k \}\Big|  = \sum\limits_{a=1}^{\frac{d-4}{2}} \sum\limits_{b=2a+3}^d \lfloor \tfrac{1}{2} (b-2a-1) \rfloor = \frac{1}{24} d (d^2-6d+8)\,.
  \end{align*}
  For the final fourth table, both cases with $b-1 \leq n$ are relevant, yielding
  \begin{align*}
    \sum\limits_{1 \leq a,b \leq d} &\Big|\{\Nroot{ij},\Proot{k\ell} \in N_a \times P_b \mid j=\ell \}\Big|  =\\
    &\sum\limits_{a=1}^{\tfrac{d-2}{2}} \sum\limits_{b=a+2}^{\min(2a+1,d)} (b-a-1)  + \sum\limits_{a=1}^{\tfrac{d-2}{2}} \sum\limits_{b=2a+2}^{d} \lceil \tfrac{b-1}{2} \rceil  =  \tfrac{1}{8} (d-2)d(d-1)\,.
  \end{align*}
  Combining these numbers of interaction with the right value of the covariance yields to
  \begin{align*}
    \Cov(\Pinv, \Ninv) &= \tfrac{1}{12} \Big(\tfrac{1}{24} d (d^2-6d+8) + \tfrac{1}{8}(d-2)d(d-1) - \tfrac{1}{4}d^2(d-2)- \tfrac{1}{4}d^2(d-2)\Big)\\
    & = -\tfrac{1}{36}d^3 + \tfrac{1 }{32} d^2 + \tfrac{7}{144}d \qedhere
  \end{align*}
\end{proof}

Some analog thoughts give us the cardinalities for other interactions of the three sets $N_a$, $P_a$ and $O_a$ as well as the covariances stated in Proposition~\ref{prop:covariancesB}.

\begin{proposition}
  \label{prop:interactionNO}
  We consider the two sets $N_a$ and $O_b$ for type $B_n$ with $1 \leq a,b \leq n$.
  Depending on satisfied relations between $n,a$ and $b$, the cardinalities of the two relavant cases for $N_a \sim O_b$ are as follows:
  
  \begin{align*}
    \Big|\{\Nroot{ij}, \Oroot{k} \in N_a \times O_b \mid i=k \}\Big| &= \left\{\begin{array}{lr}
      1, \quad 1 \leq b \leq n-a\\
      0, \quad \text{ else}.
    \end{array}\right.\\
    \Big|\{\Nroot{ij}, \Oroot{k} \in N_a \times O_b \mid j=k \}\Big| &= \left\{\begin{array}{lr}
      1, \quad a+1 \leq b \leq n\\
      0, \quad \text{ else}.
    \end{array}\right.\\
  \end{align*}
\end{proposition}

\begin{proposition}
  \label{prop:interactionPO}
  We consider the two sets $P_a$ and $O_b$ for type $B_n$ with $3 \leq a \leq 2n-1$ and $1 \leq b \leq n$.
  Depending on satisfied relations between $n,a$ and $b$, the cardinalities of the two relavant cases for $P_a \sim O_b$ are as follows:
  
  \begin{align*}
    \Big|\{\Proot{ij},\Oroot{k} \in P_a \times O_b \mid i=k \}\Big| &= \left\{\begin{array}{lr}
      1, \quad \max(1, a-n) \leq b \leq \lfloor\tfrac{a-1}{2}\rfloor\\
      0, \quad \text{ else}.
    \end{array}\right.\\
    \Big|\{\Proot{ij}, \Oroot{k} \in P_a \times O_b \mid j=k \}\Big| &= \left\{\begin{array}{lr}
      1, \quad \lfloor\tfrac{a+1}{2}\rfloor \leq b \leq \min(n, a-1)\\
      0, \quad \text{ else}.
    \end{array}\right.\\
  \end{align*}
\end{proposition}

\begin{proposition}
  \label{prop:interactionPP}
  We consider the two sets $P_a$ and $P_b$ for type $B_n$ with $3 \leq a,b \leq 2n-1$.
  Depending on satisfied relations between $n,a$ and $b$, the cardinalities of the two relavant cases for $P_a \sim P_b$ are as follows:
  \begin{align*}
    \Big|\{\Proot{ij},\Proot{kl} \in P_a \times P_b \mid i=k \}\Big| &=
    \begin{tabular}{c ||c|c}
      & \tiny{$\lfloor \tfrac{a-1}{2} \rfloor \leq \lfloor \tfrac{b-1}{2} \rfloor$} & \tiny{$\lfloor \tfrac{b-1}{2} \rfloor \leq \lfloor \tfrac{a-1}{2} \rfloor$} \\
      \hline \hline
      \tiny{$a-n, b-n \leq 1$} & $\lfloor \tfrac{a-1}{2} \rfloor$ & $\lfloor \tfrac{b-1}{2} \rfloor $ \\ \hline
      \tiny{$1, b-n \leq a-n$} & $\lfloor \tfrac{1}{2} ( 2n-a+1) \rfloor$ & $\lfloor \tfrac{1}{2}  (2n-2a+b+1)\rfloor$ \\ \hline
      \tiny{$1, a-n \leq b-n$} & $\lfloor \tfrac{1}{2}  (2n-2b+a+1)\rfloor$ &  $\lfloor \tfrac{1}{2} ( 2n-b+1) \rfloor$
    \end{tabular}\\[1em]
    \Big|\{\Proot{ij},\Proot{kl} \in P_a \times P_b \mid i=\ell \}\Big| &=
    \begin{tabular}{c ||c|c|c}
      & \tiny{$\lfloor \tfrac{a-1}{2} \rfloor \leq n, b-1$} & \tiny{$n \leq \lfloor \tfrac{a-1}{2} \rfloor, b-1$} & \tiny{$b-1 \leq \lfloor \tfrac{a-1}{2} \rfloor, n$}\\
      \hline \hline
      \tiny{$a-n, \lceil \tfrac{b+1}{2} \rceil \leq 1$} & $\lfloor \tfrac{a-1}{2} \rfloor$ & $n $ & $b-1$ \\ \hline
      \tiny{$ 1, \lceil \tfrac{b+1}{2} \rceil \leq a-n $} & $\lfloor \tfrac{1}{2} (2n-a+1) \rfloor$ & $2n-a+1$ & $n+b-a $ \\ \hline
      \tiny{$ 1, a-n \leq \lceil \tfrac{b+1}{2} \rceil$} & $\lfloor \tfrac{a-1}{2} \rfloor - \lceil \tfrac{b+1}{2} \rceil +1$ & $\lfloor \tfrac{1}{2}(2n-b+1) \rfloor$ & $\lfloor \tfrac{1}{2} (b-1) \rfloor$
    \end{tabular}\\[1em]
    \Big|\{\Proot{ij},\Proot{kl} \in P_a \times P_b \mid j=k \}\Big| &=
    \begin{tabular}{c ||c|c|c}
      & \tiny{$\lfloor \tfrac{b-1}{2} \rfloor\leq n, a-1$} & \tiny{$n \leq \lfloor \tfrac{b-1}{2} \rfloor, a-1$} & \tiny{$a-1 \leq \lfloor \tfrac{b-1}{2} \rfloor, n$}\\
      \hline \hline
      \tiny{$ 1, b-n \leq \lceil \tfrac{a+1}{2} \rceil$} & $\lfloor \tfrac{b-1}{2} \rfloor - \lceil \tfrac{a+1}{2} \rceil +1$ & $\lfloor \tfrac{1}{2} ( 2n-a+1) \rfloor $ & $ \lfloor \tfrac{1}{2} ( a-1) \rfloor $ \\ \hline
      \tiny{$ \lceil \tfrac{a+1}{2} \rceil , b-n \leq 1 $} & $\lfloor \tfrac{b-1}{2} \rfloor$ & $n$ & $a-1$ \\ \hline
      \tiny{$ 1, \lceil \tfrac{a+1}{2} \rceil \leq b-n  $} & $\lfloor \tfrac{1}{2} (2n-b+1) \rfloor $ & $2n-b+1$ & $n+a-b$
    \end{tabular}\\[1em]
    \Big|\{\Proot{ij},\Proot{kl} \in P_a \times P_b \mid i=k \}\Big| &=
    \begin{tabular}{c ||c|c|c}
      & \tiny{$n \leq a-1, b-1$} & \tiny{$a-1 \leq n, b-1$} & \tiny{$b-1 \leq n, b-1$}\\
      \hline \hline
      \tiny{$ \lceil \tfrac{b+1}{2} \rceil \leq \lceil \tfrac{a+1}{2} \rceil$} & $\lfloor \tfrac{1}{2} ( 2n-a+1) \rfloor $ & $\lfloor \tfrac{1}{2} (a-1) \rfloor  $ & $ \lfloor \tfrac{1}{2} ( 2b-a-1) \rfloor  $ \\ \hline
      \tiny{$ \lceil \tfrac{a+1}{2} \rceil \leq \lceil \tfrac{b+1}{2} \rceil $} & $\lfloor \tfrac{1}{2} ( 2n-b+1) \rfloor $ & $\lfloor \tfrac{1}{2} ( 2a-b-1) \rfloor$ & $\lfloor \tfrac{1}{2} (b-1) \rfloor$
    \end{tabular}
  \end{align*}
\end{proposition}

\begin{proposition}
  \label{prop:interactionOO}
  We consider the two sets $O_a$ and $O_b$ for type $B_n$ with $1 \leq a,b \leq n$.
  The relevant interactions of $O_a \sim O_b$ are as follows:
  
  \begin{align*}
    \Big|\{\Oroot{i},\Oroot{j} \in O_a \times O_b \mid i=j \}\Big| &= \left\{\begin{array}{lr}
      1, \quad a=b\\
      0, \quad \text{ else}.
    \end{array}\right.
  \end{align*}
\end{proposition}

Given all the covariances for different tuples of  roots one can get the variance by summing over all cases.

\begin{proof}[Proof of Theorem~\ref{thm:VarB}]
  To prove the stated variances one simply sum up all covariances for given restrictions on $d$. As an example we give the concrete summation for the first case of inversions for $d \leq \tfrac{n}{2}$ with $d$ even.
  Using Lemma~\ref{lem:cardNa}, Propositions~\ref{prop:interactionNP}, \ref{prop:interactionNO}, \ref{prop:interactionPO}, \ref{prop:interactionPP} and \ref{prop:interactionOO} leads to
  \begin{align*}
    \Var{\X_{\Phidinv{d}}} &= \Cov{(\X_{N \leq d}, \X_{N \leq d})} + 2 \Cov{(\X_{P \leq d}, \X_{N \leq d})} + 2 \Cov{(\X_{N \leq d}, \X_{O \leq d})} + \\[1em]
    & \qquad 2 \Cov{(\X_{P \leq d}, \X_{O \leq d})} + \Cov{(\X_{P \leq d}, \X_{P \leq d})} + \Cov{(\X_{O \leq d}, \X_{O \leq d})}\\[1em]
    &= \Big[\frac{1}{72} d (4d^2+3d+6n-1)\Big] + \Big[-\tfrac{1}{18} d^3 + \tfrac{1}{16} d^2 + \tfrac{7}{72} d\Big] + \Big[- \tfrac{1}{8} d^2 - \tfrac{1}{8}d \Big] + \\[1em]
    & \qquad \Big[\tfrac{1}{8}d^2 - \tfrac{1}{4} d \Big]+ \Big[\tfrac{1}{36}d^3 - \tfrac{1}{12} d^2 + \tfrac{1}{18}d \Big] + \tfrac{1}{4} d\\[1em]
    &=\frac{1}{36} d^3 + \frac{1}{48} d^2 + \Big(\frac{1}{12} n + \frac{1}{72}\Big) d \qedhere
  \end{align*}
\end{proof}

\subsection{Type $C_n$}
In this section, we provide the variances for $d$-descents and $d$-inversions of type $C_n$.

\begin{theorem}
  \label{thm:VarC}
  Let $\X_\Psi$ as in~\eqref{def:RV} and $\Phiddes{d}, \Phidinv{d}$ as in~\eqref{eq:ddesdinv}.
  We then have for $1 \leq d \leq 2n-1$ that¸
  \begin{align*}
    &\Var{\X_{\Phiddes{d}}} = \\
    &\quad \left\{\begin{array}{lr}
      \frac{1}{24} d + \frac{1}{12} n   & \text{\hspace*{267pt}if } d \leq \frac{2}{3}n  \text{, } d \text{ even} \\[1em]
      \frac{1}{24} d + \frac{1}{12} n  + \frac{1}{24}  & \text{if } d \leq \frac{n}{2} \text{, } d \text{ odd} \\[2em]
      \frac{1}{24} d + \frac{1}{12} n  & \text{if } \frac{2}{3} n \leq d \leq n \text{, } d \text{ even} \\[1em]
      \frac{1}{24} d + \frac{1}{12} n  + \frac{1}{8}  & \text{if } \frac{2}{3} n < d \leq n \text{, } d \text{ odd} \\[2em]
      -\frac{1}{8} d + \frac{1}{4} n  & \text{if } n \leq d \text{, } d \text{ even} \\[1em]
      -\frac{1}{8} d + \frac{1}{4} n  + \frac{1}{8}  & \text{if } n \leq d \text{, } d \text{ odd}
    \end{array}\right.
    \\[1em]
    &\Var{\X_{\Phidinv{d}}} = \\
    &\quad \left\{\begin{array}{lr}
      \frac{1}{36} d^3 + \frac{1}{48} d^2 + (\frac{1}{12} n + \frac{1}{72}) d & \text{if }  d \leq \frac{2}{3}n \text{, } d \text{ even} \\[1em]
      \frac{1}{36} d^3 + \frac{1}{48} d^2 + (\frac{1}{12} n + \frac{1}{72}) d + \frac{1}{48} & \text{if } d \leq \frac{2}{3}n \text{, } d \text{ odd} \\[2em]
      \frac{1}{36} d^3 + \frac{11}{96} d^2 + ( - \frac{1}{24} n + \frac{11}{144}) d + (\frac{1}{24}n^2- \frac{1}{24}) & \text{if } \frac{2}{3} n \leq d \leq n \text{, } d \text{ even} \\[1em]
      \frac{1}{36} d^3 + \frac{11}{96} d^2 + ( - \frac{1}{24} n + \frac{5}{36}) d + (\frac{1}{24}n^2- \frac{1}{12}n + \frac{5}{96}) & \text{if } \frac{2}{3} n < d \leq n \text{, } d \text{ odd} \\[2em]
      -\frac{1}{12} d^3 + (\frac{1}{3} n - \frac{13}{96}) d^2 + (-\frac{1}{3} n^2+ \frac{11}{24} n - \frac{1}{16}) d +  (\frac{1}{9} n^3 - \frac{5}{24} n^2 + \frac{7}{72}n) & \text{if } n \leq d  \text{, } d \text{ even} \\[1em]
      -\frac{1}{12} d^3 + (\frac{1}{3} n - \frac{13}{96}) d^2 + (-\frac{1}{3} n^2+ \frac{11}{24} n - \frac{1}{16}) d +  (\frac{1}{9} n^3 - \frac{5}{24} n^2 + \frac{1}{18}n + \frac{5}{96}) & \text{if } n \leq d  \text{, } d \text{ odd} \\[1em]
    \end{array}\right.
  \end{align*}
\end{theorem}

The only difference between types~$B_n$ and~$C_n$ is the changes of the heights of the roots.
In type~$C_n$, these are
\begin{align*}
  N_a  &= \big\{\Nroot{ij} \mid 1 \leq i < j \leq n \mid j-i = a \big\} \\
  O_a  &= \big\{\Oroot{i} \mid 1 \leq i \leq n \mid 2i - 1 = a \big\} \\
  P_a  &= \big\{\Proot{ij} \mid 1 \leq i < j \leq n \mid j+i - 1 = a\big\}\,
\end{align*}
so that the roots of height~$a$ are $N_a \cup O_a \cup P_a$.
All calculations are completely analogous to those for type~$B_n$.
Since these do not provide any additional insight, we suppress these tedious calculations.

\subsection{Type~$D_n$}

In this section, we finally provide the variances for $d$-descents and $d$-inversions of type~$D_{n}$.

\begin{theorem}
  \label{thm:VarD}
  Let $\X_\Psi$ as in~\eqref{def:RV} and $\Phiddes{d}, \Phidinv{d}$ as in~\eqref{eq:ddesdinv}.
  We then have for $1 \leq d \leq 2n-3$ that
  \begin{align*}
    &\Var{\X_{\Phiddes{d}}} = \\
    &\quad \left\{\begin{array}{lr}
      \frac{1}{24} d + \frac{1}{12} n  + \frac{1}{6}  & \text{\hspace*{266pt}if } d < \frac{n}{2} \text{, } d \text{ even} \\[1em]
      \frac{1}{24} d + \frac{1}{12} n  + \frac{1}{8}  & \text{if } d < \frac{n}{2} \text{, } d \text{ odd} \\[2em]
      \frac{1}{24} d + \frac{1}{12} n  + \frac{1}{3}  & \text{if } \frac{n}{2} \leq d < \frac{2}{3}n \text{, } d \text{ even} \\[1em]
      \frac{1}{24} d + \frac{1}{12} n  + \frac{7}{24}  & \text{if } \frac{n}{2} \leq d \leq \frac{2}{3}n \text{, } d \text{ odd} \\[2em]
      \frac{1}{24} d + \frac{1}{12} n + \frac{1}{6} & \text{if } \frac{2}{3} n \leq d < n \text{, } d \text{ even} \\[1em]
      \frac{1}{24} d + \frac{1}{12} n  + \frac{7}{24}  & \text{if } \frac{2}{3} n \leq d < n \text{, } d \text{ odd} \\[2em]
      -\frac{1}{8} d + \frac{1}{4} n - \frac{1}{4} & \text{if } n \leq d \text{, } d \text{ even} \\[1em]
      -\frac{1}{8} d + \frac{1}{4} n  - \frac{1}{8}  & \text{if } n \leq d \text{, } d \text{ odd}
    \end{array}\right.
    \\[1em]
    &\Var{\X_{\Phidinv{d}}} = \\
    &\quad \left\{\begin{array}{lr}
      \frac{1}{36} d^3 + \frac{1}{48} d^2 + (\frac{1}{12} n + \frac{1}{72}) d & \text{if } d < \frac{n}{2} \text{, } d \text{ even} \\[1em]
      \frac{1}{36} d^3 + \frac{1}{48} d^2 + (\frac{1}{12} n + \frac{1}{18}) d + \frac{1}{16} & \text{if } d < \frac{n}{2} \text{, } d \text{ odd} \\[2em]
      \frac{1}{36} d^3 + \frac{17}{48} d^2 + (-\frac{1}{4} n + \frac{5}{9}) d + (\frac{1}{12} n^2 - \frac{1}{4}n + \frac{1}{6}) & \text{if } \frac{n}{2} \leq d < \frac{2}{3}n \text{, } d \text{ even} \\[1em]
      \frac{1}{36} d^3 + \frac{17}{48} d^2 + (-\frac{1}{4} n + \frac{5}{9}) d + (\frac{1}{12}n^2 - \frac{1}{4}n + \frac{11}{48}) & \text{if } \frac{n}{2} \leq d < \frac{2}{3}n \text{, } d \text{ odd} \\[2em]
      \frac{1}{36} d^3 + \frac{1}{6} d^2 + \frac{13}{72} d & \text{if } \frac{2}{3} n \leq d < n \text{, } d \text{ even} \\[1em]
      \frac{1}{36} d^3 + \frac{1}{6} d^2 + \frac{11}{36} d + (-\frac{1}{12}n - \frac{1}{6}) & \text{if } \frac{2}{3} n \leq d < n  \text{, } d \text{ odd} \\[2em]
      -\frac{1}{12} d^3 + (\frac{1}{3} n - \frac{5}{12}) d^2 + (-\frac{1}{3} n^2+  n - \frac{17}{24}) d +  (\frac{1}{9} n^3 - \frac{5}{12} n^2 + \frac{13}{18}n - \frac{5}{12}) & \text{if } n \leq d  \text{, } d \text{ even} \\[1em]
      -\frac{1}{12} d^3 + (\frac{1}{3} n - \frac{5}{12}) d^2 + (-\frac{1}{3} n^2+ n - \frac{7}{12}) d +  (\frac{1}{9} n^3 - \frac{5}{12} n^2 + \frac{23}{36}n - \frac{1}{4}) & \text{if } n \leq d  \text{, } d \text{ odd} \\[1em]
    \end{array}\right.
  \end{align*}
\end{theorem}

All roots of type~$d_n$ are also roots in types~$B_n$ and~$C_n$---they again only change their heights.
In type~$D_n$, these are

\[
N_a  = \big\{\Nroot{ij} \mid 1 \leq i < j \leq n \mid j-i = a \big\}, \quad
P_a  = \big\{\Proot{ij} \mid 1 \leq i < j \leq n \mid j+i - 2 = a\big\}\,
\]
so that the roots of height~$a$ are $N_a \cup P_a$.

\medskip

Again, all calculations are completely analogous to those for type $B_n$ and are suppressed.

\end{document}